\documentclass[12pt,reqno,twoside]{amsart}
\usepackage{graphicx}
\usepackage{amssymb}
\usepackage{epstopdf}
\usepackage[asymmetric,top=3.5cm,bottom=4.3cm,left=3.1cm,right=3.1cm]{geometry}
\usepackage{booktabs} 
\usepackage{array} 
\usepackage{paralist} 
\usepackage{verbatim} 
\usepackage{subfig} 
\usepackage{tabularx}
\usepackage{amsmath,amsfonts,amsthm,mathrsfs,amssymb,cite}
\usepackage[usenames]{color}
\usepackage{bm}

\newtheorem{theorem}{Theorem}[section]

\newtheorem{lemma}{Lemma}[section]

\theoremstyle{definition}

\theoremstyle{remark}
\newtheorem{remark}{Remark}[section]

\numberwithin{equation}{section}
\DeclareMathOperator*{\argmax}{arg\,max}
\DeclareMathOperator{\curl}{curl}
\DeclareMathOperator{\Div}{div}
\DeclareMathOperator{\grad}{grad}
\DeclareMathOperator{\loc}{loc}

\title[Gesture-Computing Technique by EM waves]{On a gesture-computing technique using electromagnetic waves}

\author{Jingzhi Li}
\address{Department of Mathematics, Southern University of Science and
Technology, Shenzhen, P.~R.~China}
\email{li.jz@sustc.edu.cn}

\author{Hongyu Liu}
\address{Department of Mathematics, Hong Kong Baptist University, Kowloon Tong, Hong Kong SAR.\vspace*{-4mm}}
\address{\vspace*{-4mm}and}
\address{HKBU Institute of Research and Continuing Education, Virtual University Park, Shenzhen, P. R. China.}
\email{hongyu.liuip@gmail.com; hongyuliu@hkbu.edu.hk}

\author{Hongpeng Sun}
\address{Institute for Mathematical Sciences, Renmin University of China, Beijing, P. R. China.}
\email{hpsun@amss.ac.cn}

\begin{document}
\maketitle

\begin{abstract}

This paper is concerned with a conceptual gesture-based instruction/input technique using electromagnetic wave detection.
The gestures are modelled as the shapes of some impenetrable or penetrable scatterers from a certain admissible class, called a {\it dictionary}. The gesture-computing device generates time-harmonic electromagnetic point signals for the gesture recognition and detection. It then collects the scattered wave in a relatively small backscattering aperture on a bounded surface containing the point sources. The recognition algorithm consists of two stages and requires only two incident waves of different wavenumbers. The location of the scatterer is first determined approximately by using the measured data at a small wavenumber and the shape of the scatterer is then identified using the computed location of the scatterer and the measured data at a regular wavenumber. We provide the corresponding mathematical principle with rigorous analysis. Numerical experiments show that the proposed device works effectively and efficiently.

\medskip

\medskip

\noindent{\bf Keywords:}~~Gesture recognition, instruction/input device, electromagnetic wave propagation, inverse scattering

\noindent{\bf 2010 Mathematics Subject Classification:}~~35R30, 35P25, 78A46

\end{abstract}

\section{Introduction}

The modern technology of gesture computing and computer vision enables human beings to communicate with the machine and robots and interact easily without any mechanical devices, which can be seen as an important AI (artificial intelligence) technology; see e.g. \cite{EBNT, PST}. A gesture-based computing technology usually contains three major ingredients: the computing machine, the recognition device and the human being who gives instructions to the computer. The recognition device mainly receives and ``sees" the human body language, which are mainly hand or body gestures, and then interpret them as specific orders for the computing machine. It connects the computing machine and the human being as a bridge. By tracking the hand or body gesture, it could be possible to operate important machines easily and safely for real applications.

From our earlier discussion on the mechanism of the gesture computing technology, it is easily seen that the recognition device and the recognition method play key roles for a successful application of this technology. Nowadays, one usually utilizes cameras to capture images or videos of a human being's movements and then use computer vision and image processing techniques to recognize the gestures; see e.g. \cite{KT, MKA}, where complicated models are suggested and iterative optimization method are employed. In \cite{LWY}, the authors developed a novel gesture recognition technique via the use of acoustic waves and inverse scattering methods. To motivate the current study, we briefly discuss the function process of the new design in \cite{LWY}, and refer to \cite{LWY} for more relevant details. First, the device consists of a transmitter and an array of receivers distributed on a bounded surface containing the transmitter. The transmitter can generate time-harmonic point wave signals, and the receivers can collect wave signals around the space. The performance of the gesture shall perturb the wave propagation, leading to the so-called scattering, and the receivers then collect the scattering wave data within a relatively small backscattering aperture. The recognition process is divided into two steps. First, the transmitter generates a low-frequency wave signal and one then uses the collected scattering data to determine the location of the human being who is performing the gesture. Second, the transmitter generates a regular-frequency (compared to the size of the human being) wave signal, and one then uses the collected scattering data to determine the gesture. There are several salient features of the newly proposed technique, and in particular it does not require any lighting condition as the conventional ones of using cameras. Moreover, the computation method for the recognition process is totally ``direct" without any inversion or iteration, and hence it is very fast and robust. There is actually some engineering development Google in Project Soli of using radar waves to identify hand gestures (cf. \cite{GP}), which is stated as a new sensing technology that uses miniature radar to detect touchless gesture interactions.

In this article, we aim to further develop the gesture-computing technique in \cite{LWY} to a more practical setting of using electromagnetic (EM) waves. The propagation of EM signals are much faster than the acoustic signals. Hence, they can provide a much more timely recognitions. Moreover, the EM signals are also more sensible than the acoustic signals \cite{BW}, and they can produce more accurate gesture recognitions. Following a similar spirit to \cite{LWY}, the mathematical setup in the current study is reduced to an inverse EM scattering problem, where by emitting EM waves and collecting the corresponding scattered waves, one intends to identify the unknown scatterer (corresponding to the gestures). We would like to mention that the inverse EM scattering problems have wide applications in radar/sonar, geophysical exploration, medical imaging and remote sensing, to name a few; see e.g. \cite{BW, CK, KG, ND, RP} and the references therein. There are a few new challenges that one shall face in the design of the gesture-computing device by using the inverse EM scattering method. First, the measurement information is very limited and indeed in our design, one only has the backscattering data in a small aperture associated with two time-harmonic point signals. Second, the recognition should be conducted in a timely manner. There is a key ingredient in our study that is critical for us to overcome those challenges. It is assumed that the gestures are all from a {\it dictionary} that is known a priori. This is a reasonable and practical assumption since the admissible gestures can be captured and stored by the device in advance. A few dictionary techniques have been proposed and investigated in the literature for inverse EM scattering problems; see \cite{AA1, AA2, AA3, LLZ} and the references therein. The key ingredients of a dictionary method are the design of the appropriate dictionary class and the dictionary searching method. These are also the major technical contributions of the current article. It is noted that in the practical scenario, the human being who performs the gestures will not stand in a fixed position. Hence, one would need first to determine the location of the scatterer, namely the gesture. After that, one can use the dictionary matching algorithm to determine the specific gesture. However, in the dictionary class, the scattering information of the admissible gestures should be independent of any location requirement. This challenge can be solved by using the so-called translation relation if incident plane waves are used; see \cite{AA1, AA2, AA3, LLZ}. But in the current design, point signals are used and the scattering data are collected in a special manner. This requires some technical treatments in our study. Moreover, for timely dictionary matching, we propose a fast and robust ``direct" method based on our theoretical analysis.

The rest of the paper is organized as follows. In section \ref{sec:mathground}, we discuss the mathematical principle for the gesture computing with electromagnetic waves. In section \ref{sec:recon:algo}, we present a two-stage recognition algorithm based on the theoretical analysis. In section \ref{sec:num}, extensive numerical tests are conducted to verify the effectiveness and efficiency of the proposed algorithm. The paper is concluded with some relevant discussion in Section~\ref{sect:conclusion}.

\section{Mathematical framework}\label{sec:mathground}

In this section, we present the mathematical setting and fundamentals for the proposed gesture-computing technique.
The body shape of the person who performs gestures is modelled as a $C^2$ domain $\Omega$. $\Omega$ is assumed to have a connect complement $\Omega^c := \mathbb{R}^{3} \backslash \bar \Omega$. It is assumed that there exists a {\it dictionary} of $C^2$ domains, which could be calibrated beforehand as discussed earlier, i.e.,
\begin{equation}\label{eq:dic:class}
\mathfrak{D} = \{ D_j\}_{j=1}^{N}, \quad N \in \mathbb{N},
\end{equation}
where each $D_j$ is simply connected and contains the origin, such that
\begin{equation}\label{eq:dic:trans}
\Omega = D + z: = \{ x+z; \ x \in D \}, \quad D  \in \mathfrak{D}, \quad z \in \mathbb{R}^3.
\end{equation}
Our gesture computing strategy with EM waves could be formulated as an inverse EM scattering problem. Generally speaking, an inverse scattering problem is concerned with the recovery of an unknown scatterer by EM wave probing. To that end, one sends an incident EM field to probe the scatterer, and then measure the scattered EM wave data away from the scatterer. By using the measurement data, one can infer knowledge about the unknown scatterer. In the gesture-computing setup of the current study, the gesture $\Omega$ or the dictionary domain $D_j$ shall be modelled as a non-penetrable perfectly conducting scatterer or a penetrable medium scatterer, which already covered lots of important applications \cite{CK}. The inputs of the gesture computing are modelled as certain incident EM point waves located at a fixed spot. With the incident waves, one then measures the scattered wave due to the unknown gesture (scatterer) $\Omega$ on a measurement surface $\Gamma$ with multiple receivers. In our study, the measurement surface $\Gamma$ contains the location of the incident point waves.

In what follows, we shall need the following two assumptions,
\begin{equation}\label{eq:assume:1}
\|D_j\|: = \max_{x \in D} |x| \simeq 1, \quad 1 \leq j \leq N,
\end{equation}
together with
\begin{equation}\label{eq:assume:2}
|z| \gg 1,
\end{equation}
where $z$ is the location of $\Omega$ as in \eqref{eq:dic:trans}. Assumption \eqref{eq:assume:1} means that the size of the scatterer $\Omega$ can be calibrated such that the low frequency of the EM waves is characterized as $\frac{2 \pi}{k} \gg \|\Omega\|$ and the regular frequency scale is characterized as $\frac{2 \pi}{k} \simeq \|\Omega\|$, where $k\in\mathbb{R}_+$ signifies the wave number of the EM waves. This is practically feasible, since the frequency band and the spectrum of the  electromagnetic waves are of a very wide range \cite{BW}. Assumption \eqref{eq:assume:2} signifies that the person performing gesture instructions should stay away from the recognition device with a sufficiently large distance. Actually, we would like to point out that this condition is needed mainly for theoretical justification of the proposed gesture recognition algorithm in what follows. Indeed, in our numerical tests,
it can be seen that as long as the scatterer $\Omega$ is located away from the point sources of a reasonable distance, then the recognition algorithm works effectively and efficiently.

Now we turn to the input of our gesture computing, namely, the point incident waves for detecting the unknown gesture $\Omega$ and the plane incident waves for $D_j$ of $\mathfrak{D}$. Throughout this section, the following electric dipole with a polarization vector $p\in\mathbb{R}^3$ is chosen as an incident wave for $\Omega$ \cite{Kir1},
\begin{equation}\label{eq:point}
E_{k,ed}^{i}(x,y) = \frac{i}{k}\curl_{x} \curl_{x} [p \Phi_{k}(x,y)],\quad H_{k,ed}^{i}(x,y) = \curl_{x} [p \Phi_{k}(x,y)],
\end{equation}
where $\Phi_{k}(x,y)$ is the fundamental solution of the Helmholtz equation of wave number $k\in\mathbb{R}_+$ in $\mathbb{R}^3$ with source placed at $y$ (cf.~\cite{CK}), and the $\curl_x$ denotes the curl operator acting on $x$, i.e.,
\begin{equation}
\Phi_{k}(x,y) = \frac{e^{ik|x-y|}}{4 \pi |x-y|}, \quad x, y \in \mathbb{R}^3, \ \ x \neq y.
\end{equation}
In addition, the following incident plane wave with the polarization $p\in\mathbb{R}^3$, direction of propagation $d\in\mathbb{S}^2$ and wave number $k\in\mathbb{R}_+$ \cite{CK, ND} shall be chosen for $D_j$ of $\mathfrak{D}$,
\begin{equation}\label{eq:plane}
E_{k}^{i}(d,p;x) = ik(d\times p) \times d e^{ikx\cdot d}, \quad H_{k}^{i}(d,p;x) = ikd \times p e^{ikx \cdot d}.
\end{equation}
The following auxiliary lemma on the asymptotic relations between the electric dipole incident wave \eqref{eq:point} and the EM plane wave \eqref{eq:plane} is useful for our subsequent discussion. Henceforth, we denote $\hat{z} = z/|z|$ for $z\in\mathbb{R}^3\backslash\{0\}$.
\begin{lemma}\label{lem:incident}
The electric dipole incident wave \eqref{eq:point} and electromagnetic plane incident waves \eqref{eq:plane} have the following asymptotic transition relations,
\begin{align}
&\lim_{|z| \rightarrow \infty} E_{k,ed}^{i}(x+z,y) = \frac{e^{ik|z|-ik\hat{z} \cdot y}}{4 \pi |z|}E_{k}^{i}(\hat{z},p;x), \\
& \lim_{|z| \rightarrow \infty} H_{k,ed}^{i}(x+z,y) = \frac{e^{ik|z|-ik\hat{z} \cdot y}}{4 \pi |z|} H_{k}^{i}(\hat{z},p;x).
\end{align}
\end{lemma}
\begin{proof}
By direct calculations, together with the help of Lemma 3.29 and its proof in \cite{Kir1}, we have
\begin{align*}
E_{k,ed}^{i}(x+z,y)  &= \frac{i}{k}\curl_{x} \curl_{x} [p\Phi_{k}(x+z,y)] = \frac{i}{k}(-\Delta_{x} + \nabla_{x} \Div_{x} )[p \Phi_{k}(x+z,y)] \\
&=ik \Phi_{k}(x+z,y) \left[p-\frac{z-(y-x)}{|z-(y-x)|}\frac{[z-(y-x)]\cdot p}{|z-(y-x)|}\right] + \mathcal{O}(|z|^{-2}) \\
& = ik \frac{e^{ik|z|}}{4 \pi |z|}  e^{ik \hat{z} \cdot (x-y)}[p - \hat{z} \hat{z} \cdot p] + \mathcal{O}(|z|^{-2})  \\
& = \frac{e^{ik|z|-ik\hat{z} \cdot y}}{4 \pi |z|}  ik(\hat{z}\times p) \times \hat{z} e^{ikx\cdot \hat{z}}  + \mathcal{O}(|z|^{-2}) \\
& = \frac{e^{ik|z|-ik\hat{z} \cdot y}}{4 \pi |z|} E_{k,p}^{i}(x,\hat{z},p)  + \mathcal{O}(|z|^{-2}).
\end{align*}
Similarly, we have
\begin{align*}
H_{k,ed}^{i}(x+z,y) & = \nabla_{x} \Phi_{k}(x+z,y) \times p \\
 &= \Phi_{k}(x+z,y)\left[ik - \frac{1}{|z-(y-x)|}\right]\frac{z-(y-x)}{|z-(y-x)|}\times p \\
&  =  \frac{e^{ik|z|-ik\hat{z} \cdot y}}{4 \pi |z|} ik e^{ikx\cdot \hat{z}}  \hat{z} \times p + \mathcal{O}(|z|^{-2}) \\
 &=   \frac{e^{ik|z|-ik\hat{z} \cdot y}}{4 \pi |z|}  H_{k,p}^{i}(x) +\mathcal{O}(|z|^{-2}).
\end{align*}

The proof is complete.
\end{proof}

Next, we consider the EM scattering associated with the gesture computing described above. We shall divide our study into two separate cases: the scatterer $\Omega$ is an impenetrable conducting scatterer and is a penetrable medium scatterer.

\subsection{Perfectly Electrically Conducting Scatterer}
We first consider the case that $\Omega$ is a perfectly electrically conducting (PEC) obstacle. The electromagnetic wave scattering from a PEC scatterer in the frequency domain is governed by the following Maxwell system \cite{PM}
\begin{equation}\label{eq:scat:pec}
\begin{cases}
\curl E_{k,\Omega}^{s} -ik H_{k,\Omega}^{s} = 0, \quad \curl H_{k,\Omega}^{s} +ik E_{k,\Omega}^{s} = 0, \quad x \in \Omega^c, \\
\nu \times (E_{k,\Omega}^{s} + E_{k,ed}^i) = 0, \quad x \in \partial \Omega, \\
\displaystyle{\lim_{|x| \rightarrow \infty}(H_{k,\Omega}^{s}\times x -  |x|E_{k,\Omega}^{s} ) = 0, \ \lim_{|x| \rightarrow \infty}(E_{k,\Omega}^{s}\times x +  |x|H_{k,\Omega}^{s} )}=0,
\end{cases}
\end{equation}
where $(E_{k,\Omega}^{s}, H_{k,\Omega}^{s}) \in H_{\loc}(\curl;\Omega^c)\times H_{\loc}(\curl;\Omega^c)$ and $\nu$ signifies the exterior unit normal vector of the domain concerned. Here the last equation is called the Silver-M\"{u}ller radiation condition which captures the decaying properties the radiating scattered electromagnetic waves, and could guarantee the uniqueness of the physical solution.

We also need to consider the scattering by a PEC scatterer $D$ from $\mathfrak{D}$ by the plane incident wave \eqref{eq:plane} with the Silver-M\"{u}ller radiation condition, i.e., to find $( E_{k}^{s}(D, d,p;x)$, $H_{k}^{s}(D,d,p;x) )$ $\in H_{\loc}(\curl;D^c)$ $\times H_{\loc}(\curl;D^c)$ with $D^c = \mathbb{R}^3\backslash \bar D$, such that
\begin{equation}\label{eq:scat:pec:plane}
\begin{cases}
\curl E_{k}^{s}(D, d,p;x) -ik H_{k}^{s}(D,d,p;x) = 0,\quad x \in D^c ,\\
 \curl H_{k}^{s}(D, d,p;x) +ik E_{k}^{s}(D,d,p;x) = 0,  \quad x \in D^c, \\
\nu \times (E_{k}^{s}(D,d,p;x) + E_{k}^i(d,p;x)) = 0, \quad x \in \partial D.\\
\end{cases}
\end{equation}

For the following discussions, we also need the corresponding far field pattern of $E_{k,\Omega}^s(x)$. The far field pattern is the asymptotic amplitude of the corresponding scattered electric field or magnetic field \cite{CK}. Take $E_{k,\Omega}^{s}(x) $ and $H_{k,\Omega}^{s}(x)$ for example,
\[
E_{k,\Omega}^{s}(x) = \frac{e^{ik|x|}}{|x|}  E_{k,\Omega}^{\infty}(\hat{x}) + \mathcal{O}(|x|^{-2}), \quad H_{k,\Omega}^{s}(x) = \frac{e^{ik|x|}}{|x|}  H_{k,\Omega}^{\infty}(\hat{x})+ \mathcal{O}(|x|^{-2}),
\]
where the far field $E_{k,\Omega}^{\infty}(\hat{x})$ and $H_{k,\Omega}^{\infty}(x)$ belong to $T^{2}(\mathbb{S}^2)$ with $T^{2}(\mathbb{S}^2)$ denoting the tangential vector space of unit sphere $\mathbb{S}^2$ in $\mathbb{R}^3$ \cite{PM}.

For the scattered electromagnetic waves of system \eqref{eq:scat:pec} and \eqref{eq:scat:pec:plane}, they have the following asymptotic relations as in the following theorem, while the displacement $|z|$ is large enough.
 \begin{theorem}\label{thm:pec}
 Let $k \in \mathbb{R}_{+}$ be fixed. We have the following asymptotic expansions for the PEC scattering problem \eqref{eq:scat:pec} under the translation relation \eqref{eq:dic:trans},
 \begin{align}
 E_{k,\Omega}^s(x) &=   \frac{e^{ik|z|-ik\hat{z} \cdot y}}{4 \pi |z|} \frac{e^{ik|x-z|}}{|x-z|}[E_{k}^{\infty}(D,\hat{z},p;\widehat{x-z})  +  \mathcal{O}(|z|^{-1})][1+\mathcal{O}(|z|^{-1})], \\
 H_{k,\Omega}^s(x) &=   \frac{e^{ik|z|-ik\hat{z} \cdot y}}{4 \pi |z|} \frac{e^{ik|x-z|}}{|x-z|}[H_{k}^{\infty}(D,\hat{z},p;\widehat{x-z})  + \mathcal{O}(|z|^{-1})][1+\mathcal{O}(|z|^{-1})],
 \end{align}
 for any fixed $x \in \Omega^c$ as $|z| \rightarrow \infty$ uniformly for all $\hat{z} \in \mathbb{S}^2$, where $E_{k}^{\infty}(D,\hat{z},p;\widehat{x-z})$ and $H_{k}^{\infty}(D,\hat{z},p;\widehat{x-z})$ are the far fields of the scattered electric field and magnetic field of \eqref{eq:scat:pec:plane} scattered by incident plane wave $E_{k}^i(\hat{z},p;x)$.
 \end{theorem}
 \begin{proof}
 By Theorem 6.21 of \cite{CK}, $ E_{k,\Omega}^s(x) $ and $ H_{k,\Omega}^s(x) $ could be uniquely represented as the following integrals,
 \begin{align}
 E_{k,\Omega}^s(x) &= \curl \int_{\partial \Omega}a(y) \Phi_{k}(x,y)ds(y) + i \eta \curl \curl \int_{\partial \Omega} [\nu \times (S_{0,\Omega}^2a)(y)] \Phi_{k}(x,y)ds(y), \notag  \\
  H_{k,\Omega}^s(x) & = \frac{1}{ik} \curl  E_{k,\Omega}^s(x), \label{eq:e:intepre}
 \end{align}
 where $a(y)$ is a vector density on $\partial \Omega$, and
    $ (S_{0,\Omega}a)(x)$ is as in \cite{CK},
  \begin{equation}
  (S_{0,\Omega}a)(x):= \int_{\partial \Omega} \frac{a(y)}{|x-y|}ds(y), \ \ x \in \partial \Omega.
  \end{equation}

  Denote  $K_{k,\Omega}a: = \nu \times S_{0,\Omega}^2a$ and
 \begin{align}
 (\mathcal{S}_{k,\Omega}a)(x)&:=  \curl \int_{\partial \Omega}a(y) \Phi_{k}(x,y)ds(y), \quad x \in \Omega^c, \\
  (\mathcal{D}_{k,\Omega}a)(x) &:= \curl \curl \int_{\partial \Omega} [\nu \times (S_{0,\Omega}^2a)(y)] \Phi_{k}(x,y)ds(y), \quad x\in \Omega^c,\\
 (M_{k,\Omega}a)(x)&:= 2\int_{\partial \Omega} \nu(x)\times \curl_{x}\{a(y)\Phi_{k}(x,y)\}ds(y), \quad x\in \partial \Omega,\\
 (N_{k,\Omega}b)(x)&:= 2 \nu(x)\times \curl \curl \int_{\Omega} b(y) \Phi_{k}(x,y) ds(y),\quad x\in \partial \Omega.
 \end{align}
   By the jump relations of the vector potentials of $M_{k,\Omega}$ and $N_{k,\Omega}$ \cite{CK}, according to \eqref{eq:e:intepre}, we have
  \begin{equation}\label{eq:density:eq:Omega}
  a(x) + (M_{k,\Omega}a)(x) + i\eta (N_{k,\Omega} K_{k,\Omega}a)(x) = -2 \nu\times E_{k,ed}^i, \quad x \in \partial \Omega.
  \end{equation}
  Now we are in a position to prove this main result in the following steps. 
  
  First, let's consider $ E_{k,\Omega}^s(x+z)$, $x\in \Omega^c$. By direct calculation, we have
  \begin{align*}
  \curl_{x}[f(y)\Phi_{k}(x,y)] &= \nabla_{x}\Phi_{k}(x,y)\times f(y), \\
  \curl_{x} \curl_{x} [f(y)\Phi_{k}(x,y)] & = (-\Delta_{x} + \nabla_{x} \Div_{x})[f(y)\Phi_{k}(x,y)].
  \end{align*}
  Next it is easy to check that $ \Phi_{k}(x+z,y) = \Phi_{k}(x,y-z)$ and
  \[
  \quad \frac{\partial \Phi_{k}(\tilde x,y)}{\partial \tilde x_{j}}|_{\tilde{x} = x+z} =   \frac{\partial \Phi_{k}( x,y-z)}{\partial x_{j}},\quad   \frac{\partial^2 \Phi_{k}(\tilde x,y)}{\partial \tilde x_{j}\partial \tilde x_{i}}|_{\tilde{x} = x+z} =   \frac{\partial \Phi_{k}( x,y-z)}{\partial x_{j}\partial x_{j}}.
  \]
  Thus we have
  \begin{align}
  \curl_{x}[f(y)\Phi_{k}(x,y)](x+z) &=   \curl_{x}[f(y)\Phi_{k}(x,y-z)], \label{eq:curl:trans}\\
  \curl_{x} \curl_{x} [f(y)\Phi_{k}(x,y)](x+z) & = \curl_{x} \curl_{x} [f(y)\Phi_{k}(x,y-z)] .\label{eq:curlcurl:trans}
  \end{align}
Then with \eqref{eq:curl:trans} and \eqref{eq:curlcurl:trans}, we could write $  E_{k,\Omega}^s(x+z)$ as follows,
  \begin{align}
  E_{k,\Omega}^s(x+z) = &(\mathcal{S}_{k,\Omega}a)(x+z)+ i\eta(\mathcal{D}_{k,\Omega}a)(x+z) =  \curl \int_{\partial \Omega}a(y) \Phi_{k}(x,y-z)ds(y) \notag \\
 &+ i \eta \curl \curl \int_{\partial \Omega} [\nu \times (S_{0,\Omega}^2a)(y)] \Phi_{k}(x,y-z)ds(y). \label{eq:E:repre:Omega}
  \end{align}
   By changing variables $y = z+t$ in \eqref{eq:E:repre:Omega}, together with the assumption \eqref{eq:dic:trans}, we have
   \[
     E_{k,\Omega}^s(x+z)
      =  \curl \int_{\partial D}a(z+t) \Phi_{k}(x,t)ds(t) + i \eta \curl \curl \int_{\partial D} [\nu \times (S_{0,\Omega}^2a)(z+t)] \Phi_{k}(x,t)ds(t).
   \]
   Regarding to $ (S_{0,\Omega}a)(z+t)$, still by the change of variables, we have
   \[
   (S_{0,\Omega}a)(z+t)|_{t\in \partial D} = \int_{\partial \Omega} \frac{a(u)}{|t+z-u|}ds(u)=\int_{\partial D} \frac{a(l+z)}{|l-t|}ds(l)= (S_{0,D}a(\cdot+z))(t).
   \]
   We arrive at that
     \begin{equation}\label{eq:E:representaion:D}
     E_{k,\Omega}^s(x+z)
      =  \left([\mathcal{S}_{k,D} + i\eta\mathcal{D}_{k,D}]a(z+\cdot)\right)(x).
   \end{equation}
  
  Next, we  derive the density $a(z+t)$ with $t \in \partial D$.
   Let $x = t+z$ in \eqref{eq:density:eq:Omega} and denote $\tilde a(x)|_{x\in\partial D} = a(t+z)|_{t \in \partial D}$, still by the change of variables and similar arguments as before, we have
   \begin{equation}
     [(I + M_{k,D} + i\eta N_{k,D}K_{k,D})\tilde a](t+z)|_{\partial D} = -2 \nu\times E_{k,ed}^i(t+z,\cdot), \quad t \in \partial D,
   \end{equation}
   where here and in the following, the fixed source position in \eqref{eq:point} is omitted as in $E_{k,ed}^i(t+z,\cdot)$, and we obtain
   \begin{equation}\label{eq:a:trans:D:a}
   a(t+z) = (I + M_{k,D} + i\eta N_{k,D}K_{k,D})^{-1}(-2 \nu\times E_{k,ed}^i(t+z,\cdot)), \quad t \in \partial D.
   \end{equation}
   Substituting $a(t+z)$ as in \eqref{eq:a:trans:D:a} into \eqref{eq:E:representaion:D}, by Lemma \ref{lem:incident}, we have that for any $x \in \Omega^c$,
   \begin{align}
     &E_{k,\Omega}^s(x+z) = [\mathcal{S}_{k,D} + i\eta\mathcal{D}_{k,D}]  (I + M_{k,D} + i\eta N_{k,D}K_{k,D})^{-1}[-2 \nu\times E_{k,ed}^i(t+z,\cdot)], \notag \\
     &= \frac{e^{ik|z|-ik\hat{z} \cdot y}}{4 \pi |z|}  [\mathcal{S}_{k,D} + i\eta\mathcal{D}_{k,D}]  (I + M_{k,D} + i\eta N_{k,D}K_{k,D})^{-1}[-2 \nu\times E_{k}^{i}(\hat{z},p;t)+ \mathcal{O}(|z|^{-1})], \notag \\
     & = \frac{e^{ik|z|-ik\hat{z} \cdot y}}{4 \pi |z|}[E_{k}^s(D,\hat{z},p;x)+ \mathcal{O}(|z|^{-1})]. \label{eq:source:scatter:map}
   \end{align}
   Equation \eqref{eq:source:scatter:map} is derived from similar integral representation for the solution of the scattering by the PEC scatterer $D$ as in \eqref{eq:scat:pec:plane}. Actually, it could be checked that (see \cite{CK})
   \[
   E_{k}^s(D,\hat{z},p;x)=  [\mathcal{S}_{k,D} + i\eta\mathcal{D}_{k,D}]  (I + M_{k,D} + i\eta N_{k,D}K_{k,D})^{-1}(-2 \nu\times E_{k}^{i}(\hat{z},p;t)).
   \]
   Again by the change of variables, we have
   \begin{equation}\label{eq:near:pec}
   E_{k,\Omega}^s(x) =  \frac{e^{ik|z|-ik\hat{z} \cdot y}}{4 \pi |z|}[E_{k}^s(D,\hat{z},p;x-z)+ \mathcal{O}(|z|^{-1})].
   \end{equation}
   Thus, the theorem is proved, by the definition of far field pattern as $|x-z|$ tending to infinity while $|z|\rightarrow \infty
   $.

 \end{proof}
 It could be seen from \eqref{eq:source:scatter:map} that the mapping from the electromagnetic sources to the scattered solutions is linear. Thus, we could extend Theorem \ref{thm:pec} to multiple sources case by the following remark.
\begin{remark}\label{rem:pec}
Suppose that $E_{k,ed}^{i}(x,y_{k})$, $k = 1,2,\cdots, m$ are $m$ electric dipole sources with sources located on $y_{k}$, then the result of Theorem \ref{thm:pec} becomes
\begin{equation*}
E_{k,\Omega}^s(x) = \frac{e^{ik|z|}}{4 \pi |z|} \frac{e^{ik|x-z|}}{|x-z|} \sum_{k=1}^{m} e^{-ik\hat{z} \cdot y_{k}} [E_{k}^{\infty}(D,\hat{z},p;\widehat{x-z})  +  \mathcal{O}(|z|^{-1})][1+\mathcal{O}(|z|^{-1})].
\end{equation*}
And the near fields have the following asymptotic relations,
   \begin{equation}\label{eq:near:pec:multi}
   E_{k,\Omega}^s(x) =  \frac{e^{ik|z|}}{4 \pi |z|}\sum_{k=1}^{m}[e^{-ik\hat{z} \cdot y_{k}}E_{k}^s(D,\hat{z},p;x-z)+ \mathcal{O}(|z|^{-1})].
   \end{equation}
\end{remark}
\subsection{Electromagnetic Scattering in Inhomogeneous Medium}
Denote the refraction index as $n_{k,\Omega}(x)$, and $m_{k,\Omega} =1 -n_{k,\Omega}$. Suppose $m_{k,\Omega}$ has compact support and
$\Omega = \{x \in \mathbb{R}^3: m_{k,\Omega}(x) \neq 0\}$. Assume the real part and the imaginary part of $n_{k,\Omega}$ satisfy the following conditions \cite{CK},
\begin{equation}\label{eq:n:medium}
\Re{n_{k,\Omega}}(x) >0, \ \ \Im{n_{k,\Omega}}(x) \geq 0,  \quad x \in \Omega.
\end{equation}
For the inhomogeneous medium $D$, we still assume the translation relation \eqref{eq:dic:trans}. Thus, we have
\begin{equation}\label{eq:n:medium:D}
n_{k,\Omega}(y) = n_{k,D}(x), \quad y=x+z,  \quad m_{k,D}(x) :=  (1- n_{k,D})(x), \ \   x\in \mathbb{R}^3.
\end{equation}
The governing equations for electromagnetic medium scattering are as follows \cite{CK, RP}, i.e., to find $(E_{k,\Omega}^s, H_{k,\Omega}^s) \in H_{\loc}(\curl;\mathbb{R}^3)\times H_{\loc}(\curl; \mathbb{R}^3)$ such that \cite{PM}
\begin{equation}\label{eq:scat:medium}
\begin{cases}
\curl E_{k,\Omega} -ik H_{k,\Omega} = 0, \quad \curl H_{k,\Omega} +ik n_{k,\Omega}(x)E_{k,\Omega} = 0, \ x \neq y \in \mathbb{R}^3,\\
E_{k,\Omega} = E_{k,\Omega}^{s} + E_{k,ed}^i, \quad H_{k,\Omega} = H_{k,\Omega}^{s} + H_{k,ed}^i,  \\
\displaystyle{\lim_{|x| \rightarrow \infty}(H_{k,\Omega}^{s}\times x -  |x|E_{k,\Omega}^{s}) = 0, \quad \lim_{|x| \rightarrow \infty}(E_{k,\Omega}^{s}\times x +  |x|H_{k,\Omega}^{s} )}=0.
\end{cases}
\end{equation}
Similarly, with the same Silver-M\"{u}ller radiation condition, the scattering of inhomogeneous medium $D$ is to find $( E_{k}^{s}(D, d,p;x), H_{k}^{s}(D,d,p;x) )$ $\in H_{\loc}(\curl;\mathbb{R}^3)$ $\times$ $H_{\loc}(\curl;\mathbb{R}^3)$ such that
\begin{equation}\label{eq:scat:medium:plane}
\begin{cases}
\curl E_{k}^{s}(D, d,p;x) -ik H_{k}^{s}(D,d,p;x) = 0, \\
\curl H_{k}^{s}(D, d,p;x) +ik n_{k, D}E_{k}^{s}(D,d,p;x) = 0,
\end{cases}
\end{equation}
with
\[
E_{k,D}(x) = E_{k}^{s}(D,d,p;x) + E_{k}^i(d,p;x), \  H_{k,D}(x) = H_{k}^{s}(D,d,p;x) + H_{k}^i(d,p;x).
\]
 And for the scattered waves of system \eqref{eq:scat:medium} and \eqref{eq:scat:medium:plane}, we also have the following asymptotic relations by the following theorem, while the displacement $|z|$ is large enough.

\begin{theorem}\label{thm:medium}
 Let $k \in \mathbb{R}_{+}$ be fixed. We have the following asymptotic expansions for electromagnetic medium scattering problem \eqref{eq:scat:medium} under translation relation \eqref{eq:dic:trans},
 \begin{align}
 E_{k,\Omega}^s(x) &=   \frac{e^{ik|z|-ik\hat{z} \cdot y}}{4 \pi |z|} \frac{e^{ik|x-z|}}{|x-z|}[E_{k}^{\infty}(D,\hat{z},p;\widehat{x-z})  +  \mathcal{O}(|z|^{-1})][1+\mathcal{O}(|z|^{-1})], \\
 H_{k,\Omega}^s(x) &=   \frac{e^{ik|z|-ik\hat{z} \cdot y}}{4 \pi |z|} \frac{e^{ik|x-z|}}{|x-z|}[H_{k}^{\infty}(D,\hat{z},p;\widehat{x-z})  + \mathcal{O}(|z|^{-1})][1+\mathcal{O}(|z|^{-1})],
 \end{align}
 for any fixed $x \in \mathbb{R}^3$ as $|z| \rightarrow \infty$ uniformly for all $\hat{z} \in \mathbb{S}^2$, where $E_{k}^{\infty}(D,\hat{z},p;\widehat{x-z})$ and $H_{k}^{\infty}(D,\hat{z},p;\widehat{x-z})$ are the far fields of the scattered electric and magnetic fields of \eqref{eq:scat:medium:plane} scattered by incident plane wave $E_{k}^{i}(\hat{z},p;x)$.
 \end{theorem}

\begin{proof}
By theorem 9.1 of \cite{CK}, the scattering field has the following
integral representation, with notation $E_{k,\Omega}^{s} := T_{k,\Omega}E_{k,\Omega}$ defined as follows
\begin{equation}\label{eq:repre:medium}
  -k^2 \int_{\Omega}\Phi_{k}(x,y)m_{k,\Omega}(y)E_{k,\Omega}(y)dy  + \grad \int_{\Omega}\frac{1}{n_{k,\Omega}(y)}\grad n_{k,\Omega}(y) \cdot E_{k,\Omega}(y) \Phi_{k}(x,y)dy.
\end{equation}
By theorem 2.42 of \cite{RP}, \eqref{eq:repre:medium} also could be used for scattering problems with electric dipole source incident wave when the source point $y \in \Omega^c$. And the operator $(I -T_{k,\Omega} )$ is continuously invertible in $C(\Omega)$ while $n_{k,\Omega} \in C^{1,\alpha}(\mathbb{R}^3)$ with $0<\alpha <1$. We have
\begin{equation}\label{eq:eq:elec}
E_{k,\Omega}  = (I - T_{k,\Omega})^{-1}E_{k,ed}^{i},\quad x \in \Omega.
\end{equation}
Considering $E_{k,\Omega}^s(x+z)$, with \eqref{eq:repre:medium}, it could be written as
\begin{align}
E_{k,\Omega}^s(x+z)=&-k^2 \int_{\Omega}\Phi_{k}(x+z,y)m_{k,\Omega}(y)E_{k,\Omega}(y)dy \notag \\
 &+ \grad \int_{\Omega}\frac{1}{n_{k,\Omega}(y)}\grad n_{k,\Omega}(y) \cdot E_{k,\Omega}(y) \Phi_{k}(x+z,y)dy,\notag \\
 =&  -k^2 \int_{\Omega}\Phi_{k}(x,y-z)m_{k,\Omega}(y)E_{k,\Omega}(y)dy  \notag \\
&+ \grad \int_{\Omega}\frac{1}{n_{k,\Omega}(y)}\grad n_{k,\Omega}(y) \cdot E_{k,\Omega}(y) \Phi_{k}(x,y-z)dy. \notag
\end{align}
Setting $t=y-z$ and denoting $E_{k,D}(t+z)|_{t \in D}:= E_{k,\Omega}(y)|_{y \in \Omega}$, by the change of variables and noting that the Jacobian matrix of the change of variables is the identity matrix in $\mathbb{R}^3$, together with \eqref{eq:n:medium:D} and \eqref{eq:n:medium}, we have
\begin{align}
E_{k,\Omega}^s(x+z)  =&  -k^2 \int_{D}\Phi_{k}(x,t)m_{k,D}(t)E_{k,D}(t+z)dy  \\
                    &+ \grad \int_{D}\frac{1}{n_{k,D}(t)}\grad n_{k,D}(t) \cdot E_{k,D}(t+z) \Phi_{k}(x,t)dy. \notag
\end{align}
Then we obtain
\begin{equation}\label{eq:e:omega:z}
E_{k,\Omega}^s(x+z)  =  (T_{k,D} E_{k,D}(\cdot+z))(x).
\end{equation}
Then we turn to calculating $E_{k,D}(t+z)$. By \eqref{eq:eq:elec}, we see $(I - T_{k,\Omega})E_{k,\Omega}(x)  = E_{k,ed}^{i}(x)$, i.e.,
\begin{align}
E_{k,ed}^{i}(x)&=E_{k,\Omega}(x) + k^2 \int_{\Omega}\Phi_{k}(x,y)m_{k,\Omega}(y)E_{k,\Omega}(y)dy \notag   \\
 &- \grad \int_{\Omega}\frac{1}{n_{k,\Omega}(y)}\grad n_{k,\Omega}(y) \cdot E_{k,\Omega}(y) \Phi_{k}(x,y)dy. \label{eq:e:omega:medium}
\end{align}
Again setting $x = t+z$ and by the change of variables, \eqref{eq:e:omega:medium} becomes
\begin{align}
E_{k,ed}^{i}(t+z)&-E_{k,D}(t+z) =k^2 \int_{\Omega}\Phi_{k}(t+z,y)m_{k,\Omega}(y)E_{k,\Omega}(y)dy  \notag \\
&- \grad \int_{\Omega}\frac{1}{n_{k,\Omega}(y)}\grad n_{k,\Omega}(y) \cdot E_{k,\Omega}(y) \Phi_{k}(t,y-z)dy, \notag \\
  &= k^2 \int_{\Omega}\Phi_{k}(t,y-z)m_{k,\Omega}(y)E_{k,\Omega}(y)dy \label{eq:mega:22}\\
& - \grad \int_{\Omega}\frac{1}{n_{k,\Omega}(y)}\grad n_{k,\Omega}(y) \cdot E_{k,\Omega}(y) \Phi_{k}(t,y-z)dy. \label{eq:mega:2}
\end{align}
Still using changing variable $\tilde y = y-z$, \eqref{eq:mega:22} and \eqref{eq:mega:2} could be written as
\begin{align}
&k^2 \int_{D}\Phi_{k}(t,\tilde y)m_{k,D}(\tilde y)E_{k,D}(\tilde y +z)d\tilde y \notag  \\
&- \grad \int_{D}\frac{1}{n_{k,D}(\tilde{y})}\grad n_{k,D}(\tilde y) \cdot E_{k,D}(\tilde y +z )\Phi_{k}(t,\tilde{y})d \tilde y. \notag
\end{align}
What follows is
\begin{equation}
E_{k,D}(t+z) = (I - T_{k,D})^{-1}E_{k,ed}^{i}(t+z,\cdot), \quad t \in D.
\end{equation}
Substituting it into \eqref{eq:e:omega:z}, we have
\begin{equation}
E_{k,\Omega}^s(x+z)  =  [T_{k,D} (I - T_{k,D})^{-1}E_{k,ed}^{i}(t+z,\cdot)](x).
\end{equation}
Then by Lemma \ref{lem:incident}, we arrive at that
\begin{align}
E_{k,\Omega}^s(x+z)  &=  \frac{e^{ik|z|-ik\hat{z} \cdot y}}{4 \pi |z|}[\left(T_{k,D} (I - T_{k,D})^{-1}E_{k}^{i}(\hat{z},p;t)\right)(x) + \mathcal{O}(|z|^{-1})], \label{eq:map:source:scatter:medium} \\
& = \frac{e^{ik|z|-ik\hat{z} \cdot y}}{4 \pi |z|}[E_{k}^s(D,\hat{z},p;x)+ \mathcal{O}(|z|^{-1})]. \notag
\end{align}
Actually for the medium $D$ scattered by the plane wave \eqref{eq:scat:medium:plane}, we have (see \cite{CK})
\[
E_{k}^s(D,\hat{z},p;x) = \left[T_{k,D} (I - T_{k,D})^{-1}(E_{k}^{i}(\hat{z},p;\cdot))\right](x).
\]
Still by the change of variables, we have
\begin{equation}\label{eq:nearfield:re}
E_{k,\Omega}^s(x) = \frac{e^{ik|z|-ik\hat{z} \cdot y}}{4 \pi |z|}[E_{k}^s(D,\hat{z},p;x-z)+ \mathcal{O}(|z|^{-1})].
\end{equation}
The theorem is proved by taking the far field pattern as $|x-z|$ tends to infinity.
\end{proof}

Similar to Remark \ref{rem:pec}, due to the linearity of the mapping $T_{k,D} (I - T_{k,D})^{-1}$ from the electromagnetic sources to the scattered solutions as in \eqref{eq:map:source:scatter:medium}, we could extend Theorem \ref{thm:medium} to multiple sources case by the following remark.

\begin{remark}\label{rem:medium}
 Suppose that $E_{k,ed}^{i}(x,y_{k})$, $k = 1,2,\cdots, m$ are $m$ electric dipole sources that with sources located on $y_{k}$, then the result of Theorem \ref{thm:pec} becomes
\begin{equation*}
E_{k,\Omega}^s(x) = \frac{e^{ik|z|}}{4 \pi |z|} \frac{e^{ik|x-z|}}{|x-z|} \sum_{k=1}^{m} e^{-ik\hat{z} \cdot y_{k}} [E_{k}^{\infty}(D,\hat{z},p;\widehat{x-z})  +  \mathcal{O}(|z|^{-1})][1+\mathcal{O}(|z|^{-1})].
\end{equation*}
And for the near fields, we have the following asymptotic relation
\begin{equation}\label{eq:nearfield:re:multi}
E_{k,\Omega}^s(x) = \frac{e^{ik|z|}}{4 \pi |z|}\sum_{k=1}^{m}[e^{-ik\hat{z} \cdot y_{k}}E_{k}^s(D,\hat{z},p;x-z)+ \mathcal{O}(|z|^{-1})].
\end{equation}
\end{remark}

\section{A two-stage reconstruction algorithm}\label{sec:recon:algo}
\subsection{Location Determination}\label{subsec:location}
We will present a two-stage algorithm for gesture recognition using electromagnetic waves, i.e., locating the positions and determining the \emph{gestures} of the scatterers from the dictionary. We will first locate the scatterers with low frequency scattered field by sending low frequency electromagnetic point source waves. Then we determine the \emph{gestures} of the scatterers by regular frequency scattered field by sending regular frequency electromagnetic incident point source. According to Theorems \ref{thm:pec} and \ref{thm:medium}, all scattered fields produced by incident point sources for $\Omega$ could be approximated by the far field data in a precomputed dictionary, namely a database produced by scattering amplitude of incident plane waves impinging upon the translated scatterer $D$. The computations can be carried out beforehand and are collected in the precomputed gesture dictionary. Here and in the following, we assume there is only one point source located at the origin, i.e., $y_1=0$ with $m=1$ as in Remarks \ref{rem:pec} and \ref{rem:medium}.

We need the following theorem first, which is a classical result for the low frequency asymptotics for electromagnetic wave scattering problems, see chapter 3 of \cite{DR}.
\begin{theorem}\label{thm:lowfre:asym}
With the plane incident wave \eqref{eq:plane}, the far field pattern $E_{k}^{\infty}(D,d,p;\hat{x})$ and $H_{k}^{\infty}(D,d,p;\hat{x})$
have the following asymptotic behavior, no matter $D$ is a PEC scatterer as in \eqref{eq:scat:pec:plane} or an inhomogeneous medium as in \eqref{eq:scat:medium:plane},
\begin{align}
E_{k}^{\infty}(D,d,p;\hat{x}) &= \frac{(ik)^3}{4 \pi}[\hat{x} \times (\hat{x}\times \textbf{a}(D)) - \hat{x}\times \textbf{b}(D)] + \mathcal{O}(k^4), \\
H_{k}^{\infty}(D,d,p;\hat{x}) &= \frac{(ik)^3}{4 \pi}[\hat{x} \times (\hat{x}\times \textbf{c}(D)) - \hat{x}\times \textbf{d}(D)] + \mathcal{O}(k^4),
\end{align}
where $\textbf{a}(D)$, $\textbf{b}(D)$, $\textbf{c}(D)$, $\textbf{d}(D)$ are constant vectors that only depend on $p$, $d$, $D$.
\end{theorem}
Theorem \ref{thm:lowfre:asym} characterizes the most important and leading terms in the low frequency asymptotic analysis, which we would use to design efficient indicators.  With Theorems \ref{thm:pec} and \ref{thm:medium}, letting $|x| \rightarrow \infty$,  we can get the far field pattern of scattered field $E_{k,\Omega}^{s}(x)$ scattering from the corresponding PEC scatterer or inhomogeneous medium,
\begin{equation}\label{eq:asym:trans:far:ele}
E_{k,\Omega}^{\infty}(\hat{x}) = \frac{e^{ik|z|}}{4 \pi |z|} e^{-ik\hat{x}\cdot z}  [E_{k}^{\infty}(D,\hat{z},p_{k};\widehat{x-z})  +  \mathcal{O}(|z|^{-1})][1+\mathcal{O}(|z|^{-1})],
\end{equation}
with the far field $E_{k,\Omega}^{\infty}(\hat{x})$  belonging to $T^{2}(\mathbb{S}^2)$.

Now we can give the first stage algorithm for locating the scatterer $\Omega = z+D$. Here, different from the acoustic case \cite{LWY}, we use the far field instead.
Denote $E_{k}^{\infty}(D,z;\hat{x})=E_{k,\Omega}^{\infty}(\hat{x})$, and introduce
\[
\mathring{E}_{k, H_1}^{\infty}(\tilde{z};\hat{x}) = \frac{e^{ik |\tilde{z}|}}{4\pi |\tilde{z}|} e^{-ik \hat{x}\cdot \tilde{z}}U_{1}^{m}(\hat{x}),\quad \mathring{E}_{k, H_2}^{\infty}(\tilde{z};\hat{x}) = \frac{e^{ik |\tilde{z}|}}{4\pi |\tilde{z}|} e^{-ik \hat{x}\cdot \tilde{z}}V_{1}^{m}(\hat{x}), \quad m=-1,0,1,
\]
where $U_{1}^m$ and $V_{1}^m$ are the vector spherical harmonics \cite{CK},
\[
U_{1}^m(\hat{x}): = \frac{1}{2} \text{Grad}Y_{n}^m(\hat{x}), \quad V_{1}^m(\hat{x}): = \frac{1}{2} \hat{x} \times \text{Grad}Y_{n}^m(\hat{x}), \quad m=-1,0,1.
\]
We propose the following indicator function,
\begin{equation}\label{eq:indicator}
I_{k}(D,z, \tilde{z}): = \frac{\sqrt{\sum_{j=1}^2\sum_{m=-1}^1|\langle E_{k}^{\infty}(D,z;\cdot), \mathring{E}_{k,m, H_j}^{\infty}(\tilde{z};\cdot)\rangle_{T^{2}(\mathbb{S}^2)}|^2}}{\|E_{k}^{\infty}(D,z;\cdot)\|_{T^2(\mathbb{S}^2)} 1/{(4\pi |\tilde{z}|)}}, \quad \tilde{z} \in S,
\end{equation}
Here and in the following, $S$ denotes the set of sampling points for locating the position of gesture $\Omega$ and we assume $S$ includes the position of $\Omega$. A more practical indicator function is
\begin{equation}\label{eq:indicator:gamma}
I_{k,\Gamma}(D,z, \tilde{z}): = \frac{\sqrt{\sum_{j=1}^2\sum_{m=-1}^1|\langle E_{k}^{\infty}(D,z;\cdot), \mathring{E}_{k,m, H_j}^{\infty}(\tilde{z};\cdot)\rangle_{T^{2}({\Gamma})}|^2}}{\|E_{k}^{\infty}(D,z;\cdot)\|_{T^2({\Gamma})} 1/{(4\pi |\tilde{z}|)}}, \quad \tilde{z} \in S,
\end{equation}
where $\Gamma$ is part of $\mathbb{S}^2$.
\begin{theorem}\label{thm:asym:in}
Let $\textbf{a}(D)$ and $\textbf{b}(D)$ be given in Theorem \ref{thm:lowfre:asym} ($D$ is a PDE scatterer or a medium), assuming
\[
\hat{x} \times (\hat{x}\times \textbf{a}(D)) - \hat{x}\times \textbf{b}(D) \neq 0,\quad \forall D \in \mathfrak{D},
\]
then we have the following asymptotic expansions
\begin{equation}
\lim_{k\rightarrow 0}I_{k}(D, z;\tilde{z}) = \mathring{I}_{k}(z,\tilde{z})[1+\mathcal{O}(|z|^{-1})], \quad |z| \rightarrow \infty,
\end{equation}
uniformly for all $D \in \mathfrak{D}$, $\hat{z}\in \mathbb{S}^2$ and $\tilde{z}\in S$, where
\begin{equation}\label{eq:indicator:asym}
\mathring{I}_{k}(z, \tilde{z}): = \frac{\sqrt{\sum_{j=1}^2\sum_{m=-1}^1|\langle \tilde{E}_{k}^{\infty}(D,z;\cdot), \mathring{E}_{k,m, H_j}^{\infty}(\tilde{z};\cdot)\rangle_{T^{2}(\mathbb{S}^2)}|^2}}{\|\tilde{E}_{k}^{\infty}(D,z;\cdot)\|_{T^2(\mathbb{S}^2)} 1/{(4\pi |\tilde{z}|)}}, \quad \tilde{z} \in S.
\end{equation}
Here
\[
\tilde{E}_{k}^{\infty}(D,z;\cdot) =  \frac{e^{ik|z|}}{4 \pi |z|} e^{-ik\hat{x}\cdot z} (\sum_{j=-1}^1a_{1}^j U_{1}^{j}(\hat{x}) +
\sum_{j=-1}^1 b_{1}^j V_{1}^{j}(\hat{x})),
\]
and the constants $a_{1}^j$ and $b_{1}^j$ only depend on $\textbf{a}(D)$ and $\textbf{b}(D)$.
 The unique maximum of $\mathring{I}_{k}(z, \tilde{z})$ is obtained at $\tilde{z}=z $ with maximum value 1.
\end{theorem}
\begin{proof}
By equation \eqref{eq:asym:trans:far:ele} and Theorem \ref{thm:lowfre:asym}, we see
\begin{equation}
E_{k,\Omega}^{\infty}(\hat{x}) = \frac{e^{ik|z|}}{4 \pi |z|} e^{-ik\hat{x}\cdot z} (\sum_{j=-1}^1a_{1}^j U_{1}^{j}(\hat{x}) +
\sum_{j=-1}^1 b_{1}^j V_{1}^{j}(\hat{x}))[1+\mathcal{O}(|z|^{-1}) + \mathcal{O}(|x|^{-1})].
\end{equation}
Substituting it into $I_{k}$ in \eqref{eq:indicator}, we get \eqref{eq:indicator:asym}. And since
\[
e^{-ik \hat{x}\cdot \tilde{z}}U_{1}^{m}(\hat{x}), \quad  e^{-ik \hat{x}\cdot \tilde{z}}V_{1}^{m}(\hat{x}),\quad m=-1,0,1,
\]
are normalized and are orthogonal to each other in $T^2(\mathbb{S}^2)$, by Cauchy-Schwarz inequality $\langle a,b \rangle_{T^2(\mathbb{S}^2)}  \leq \|a\|_{T^2(\mathbb{S}^2)}\|b\|_{T^2(\mathbb{S}^2)}$, we have
\[
\sum_{j=1}^2\sum_{m=-1}^1|\langle \tilde{E}_{k}^{\infty}(D,z;\cdot), \mathring{E}_{k,m, H_j}^{\infty}(\tilde{z};\cdot)\rangle_{T^{2}(\mathbb{S}^2)}|^2 \leq \frac{1}{(4\pi |\tilde{z}|)^2} \|\tilde{E}_{k}^{\infty}(D,z;\cdot)\|_{T^2(\mathbb{S}^2)}^2,
\]
which leads to $\mathring{I}_{k}(z, \tilde{z}) \leq 1$.
\end{proof}

Thus we could locate the gesture (scatterer) by finding
\begin{equation}\label{eq:argmax:position}
\mathring{z}= \argmax_{\tilde{z}}I_k(D, z;\tilde{z}), \quad \tilde{z} \in S,
\end{equation}
as in \eqref{eq:indicator} on the sampling set $S$ with low frequency data.
Here, the far field pattern are employed instead of near field for locating the scatterer.
The far field could be approximately measured by the near filed around more than 10 wavelength away, which is feasible. For a timely recognition, we only make use of point source waves of wavenumber $k \lesssim 1$.
\subsection{Shape Determination}
After determining the location $\mathring{z}$ through \eqref{eq:argmax:position} of the scatterer, we will present the second stage algorithm for determining the shape of the scatterers with dictionary data. With Theorems \ref{thm:pec}, \ref{thm:medium}, \ref{thm:lowfre:asym} and \ref{thm:asym:in}, we give the following indicator functionals,
\begin{equation}\label{eq:indicator:J}
J_{k}(D_i, D_j;z, \mathring{z}): = \frac{|\langle E_{k}^{\infty}(D_i,z;\hat{x}),  \hat{E}_{k}^{\infty}(D_j,\mathring{z};\hat{x})\rangle_{T^2(\mathbb{S}^2)} | }{\|E_{k}^{\infty}(D_i,z;\hat{x})\|_{T^2(\mathbb{S}^2)} \|\hat{E}_{k}^{\infty}(D_j,\mathring{z};\hat{x})\|_{T^2(\mathbb{S}^2)} },
\end{equation}
\begin{equation}
\hat{E}_{k}^{\infty}(D_j,\mathring{z};\hat{x}): = \frac{e^{ik|\mathring{z}|}}{4 \pi |\mathring{z}|} e^{-ik\hat{x}\cdot \mathring{z}}  E_{k}^{\infty}(D_j,\hat{\mathring{z}},p_{k};\widehat{x-\mathring{z}}),
\end{equation}
where $E_{k}^{\infty}(D_i,z;\hat{x}) = E_{k,D_i+z}^{\infty}(\hat{x})$ as in \eqref{eq:asym:trans:far:ele} and $\hat{\mathring{z}}: = \mathring{z}/|\mathring{z}|$.
With these preparation, we could present our second stage algorithm for the shape determinations.
Through the following scheme, we could find the shape of $D$ by the dictionary data.
\begin{theorem}\label{thm:shape:deter}
Suppose there exists a constant $c_0>0$ such that $\|E_{k}^{\infty}(D_i,z;\hat{x})\|_{T^2(\mathbb{S}^2)}  \geq c_0$  for all $D_i \in \mathfrak{D}$. The for any sufficient small $\varepsilon>0$ there exists $R_0$ and $\sigma>0$ such that
if $|z| \geq R_0$ and $|z-\mathring{z}| \leq \sigma$,
\begin{equation}\label{eq:far:indica2:near}
|J_{k}(D_i,D_j;z, \mathring{z}) - \hat{J}_{k}(D_i,D_j;z)| \leq \varepsilon, \quad \forall D_i, D_j \in \mathfrak{D},
\end{equation}
where
\[
\hat{J}_{k}(D_i,D_j;z): = J_{k}(D_i,D_j;z,z).
\]
If further assume $E_{k}^{\infty}(D_i,z;\hat{x})$ and $E_{k}^{\infty}(D_j,z;\hat{x})$ are linearly independent for all $D_i, D_j \in \mathfrak{D}$, $i\neq j$, then we have
\begin{equation}\label{eq:indi22:sepa}
J_{k}(D_i,D_i;z, \mathring{z}) > J_{k}(D_i,D_j;z, \mathring{z}), \quad \forall i \neq j.
\end{equation}
\end{theorem}
\begin{proof}
Consider the PEC scatterer for example, while the inhomogeneous medium case is similar.
By Theorem \ref{thm:pec} and the corresponding far field pattern as in \eqref{eq:asym:trans:far:ele}, for any fixed and small $\epsilon$, there exists $R>0$ such that while $|z| > R_0$, we have
\begin{equation}\label{eq:1:E:E1}
\|E_{k}^{\infty}(D_j,z;\hat{x}) - \hat{E}_{k}^{\infty}(D_j, z;\hat{x})\|_{T^2(\mathbb{S}^2)} \leq \epsilon, \quad \forall D_j.
\end{equation}
Furthermore in light of the analytic continuity of the far field pattern \cite{CK}, there exists some small constant $\sigma>0$ such that we have
\begin{equation}\label{eq:2:E:ring}
\|\hat{E}_{k}^{\infty}(D_j, z;\hat{x}) - \hat{E}_{k}^{\infty}(D_j,\mathring{z};\hat{x})\|_{T^2(\mathbb{S}^2)} \leq \epsilon, \quad \forall D_j,
\end{equation}
whenever $|z-\mathring{z}| \leq \sigma$.
 
Combining \eqref{eq:1:E:E1} and \eqref{eq:2:E:ring} and noticing that both the inner product $\langle E_{k}^{\infty}(D_i,z;\hat{x}), \cdot \rangle_{T^2(\mathbb{S}^2)}$ and the norm $\|\cdot\|_{T^2(\mathbb{S}^2)}$ in the definitions of $J_k$ and $\hat{J}_k$ are continuous respecting to the space $T^2(\mathbb{S}^2)$, we could arrive at \eqref{eq:far:indica2:near} by standard mathematical analysis. In addition, \eqref{eq:indi22:sepa} could be directly verified  by the Cauchy-Schwarz inequality.

\end{proof}

Actually, for the shape determination, with \eqref{eq:near:pec:multi} in Remark \ref{rem:pec} and
\eqref{eq:nearfield:re:multi} in Remark \ref{rem:medium}, it could be seen that the near field could also be used. And we could get the following theorem as Theorem \ref{thm:shape:deter} by replacing the far fields with near fields in the assumptions. The proof is quite similar as Theorem \ref{thm:shape:deter} and we omit it here.

\begin{theorem}\label{thm:shape:deter:bdry}
The following indicator function could be used if the dictionary data include near fields scattered data.
\begin{equation}\label{eq:indicator:J:gamma}
J_{k}^s(D_i, D_j;z, \mathring{z}): = \frac{|\langle E_{k}^{s}(D_i,z;x),  \hat{E}_{k}^{s}(D_j,\mathring{z};x)\rangle_{L^{2}(\Gamma)} | }{\|E_{k}^{s}(D_i,z;x)\|_{L^{2}(\Gamma)} \|\hat{E}_{k}^{s}(D_j,\mathring{z};x)\|_{L^{2}(\Gamma)}}
\end{equation}
\begin{equation}
\hat{E}_{k}^{s}(D_j,\mathring{z};x): = \frac{e^{ik|\mathring{z}|}}{4 \pi |\mathring{z}|} E_{k}^s(D,\hat{\mathring{z}},p_{k};x-\mathring{z}),
\end{equation}
where $E_{k}^{s}(D_i,z;x) = E_{k,D_i+z}^{s}(x)$ and $\hat{\mathring{z}}: = \mathring{z}/|\mathring{z}|$, $\Gamma$ is a bounded measurement surface in $\mathbb{R}^3$ of finite aperture,
and $E_{k}^s(D,\hat{\mathring{z}},p_{k};x-\mathring{z})$ is the scattered near field of scatterer $D$ measured and stored in the dictionary data set.

Then we can reconstruct the shape by near field data as follows.
Suppose there exists a constant $c_1>0$ such that $\|E_{k}^{s}(D_i,z;x)\|_{L^{2}(\Gamma)}  \geq c_1$  for all $D_i \in \mathfrak{D}$. The for any sufficient small $\varepsilon_1>0$ there exists $R_1$ and $\sigma_1>0$ such that
if $|z| \geq R_1$ and $|z-\mathring{z}| \leq \sigma_1$,
\begin{equation}
|J_{k}^s(D_i,D_j;z, \mathring{z}) - \hat{J}_{k}^s(D_i,D_j;z)| \leq \varepsilon_1, \quad \forall D_i, D_j \in \mathfrak{D},
\end{equation}
where
\[
\hat{J}_{k}^s(D_i,D_j;z): = J_{k}^s(D_i,D_j;z,z).
\]
If further assume $E_{k}^{s}(D_i,z;x)$ and $E_{k}^{s}(D_j,z;x)$ are linearly independent for all $D_i, D_j \in \mathfrak{D}$, $i\neq j$, then we have
\begin{equation}
J_{k}^s(D_i,D_i;z, \mathring{z}) > J_{k}^s(D_i,D_j;z, \mathring{z}), \quad \forall i \neq j.
\end{equation}
\end{theorem}

\section{Numerical experiments and discussions}\label{sec:num}

In this section, we present numerical experiments to illustrate the effectiveness and efficiency of the proposed
recovery method, which has been successfully employed in gesture recognition using full aperture far field data or limited aperture near field data.
All the numerical experiments are carried out using MATLAB R2017a on a
Lenovo workstation with 2.3GHz Intel Xeon E5-2670 v3 processor and 512GB of RAM.

\begin{figure}
\hfill{}\includegraphics[width=0.48\textwidth]{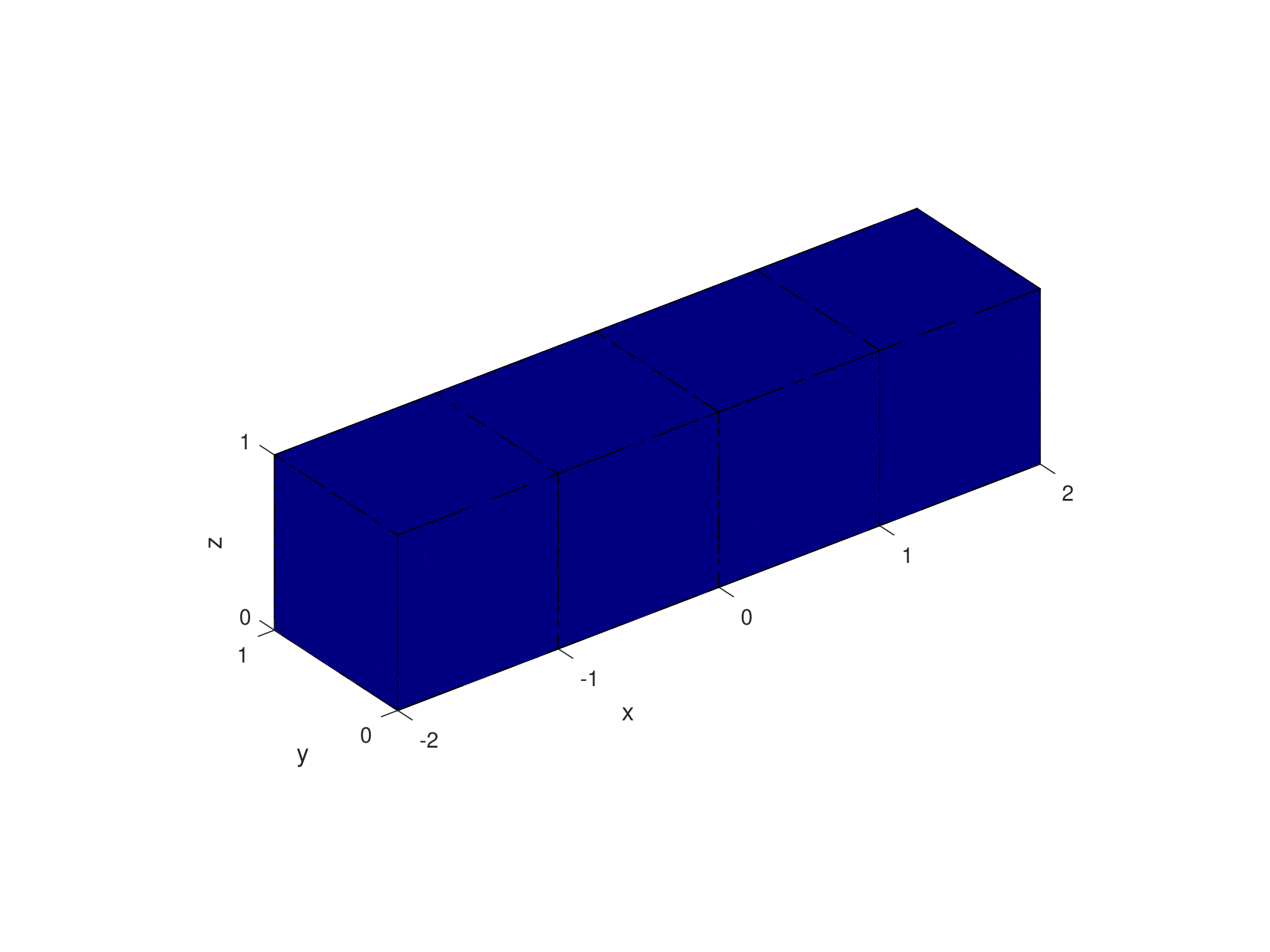}\hfill{}\includegraphics[width=0.48\textwidth]{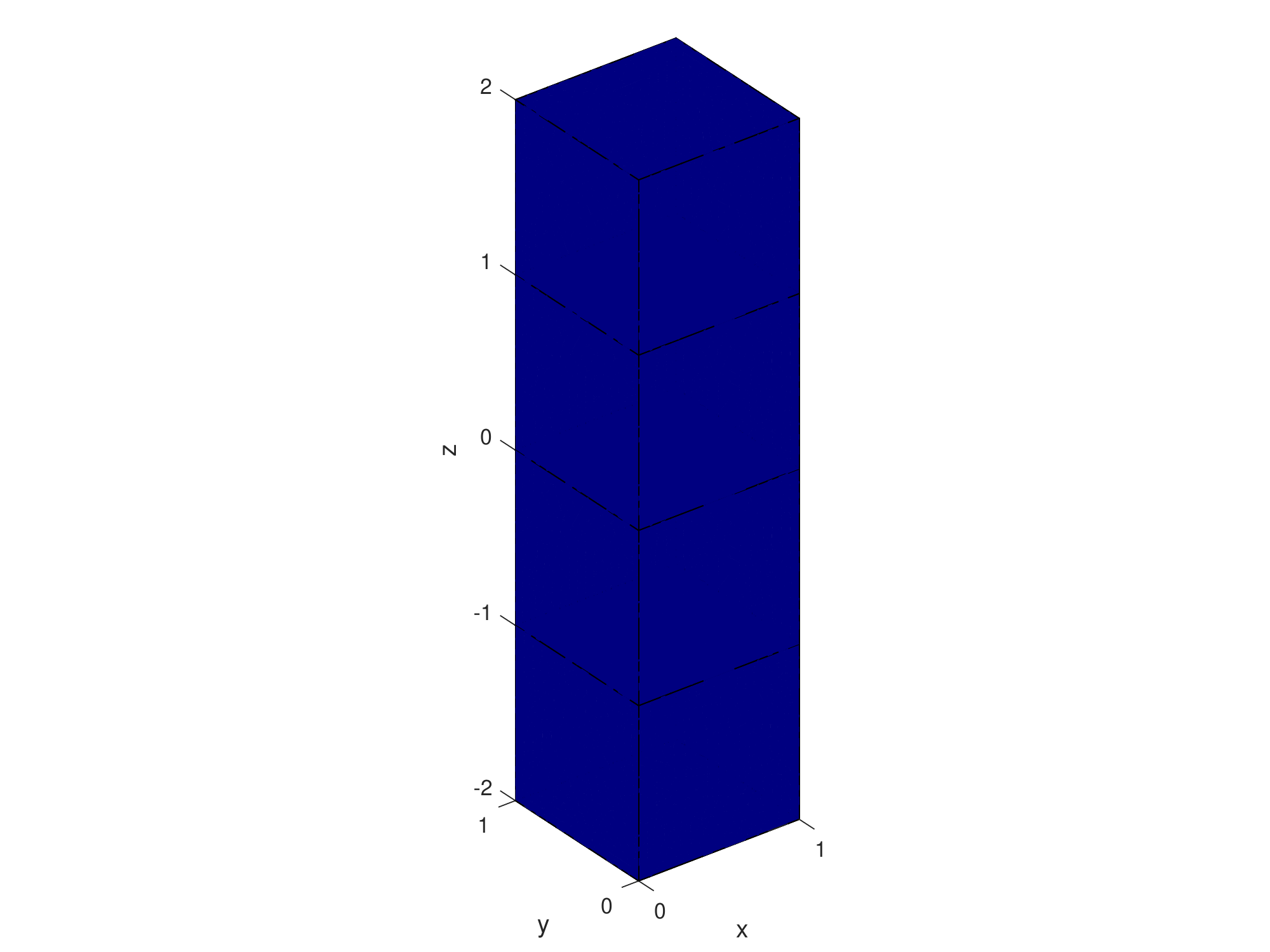}\hfill{}

\hfill{}(a)\hfill{}\hfill{}
(b)\hfill{}

\hfill{}\includegraphics[width=0.48\textwidth]{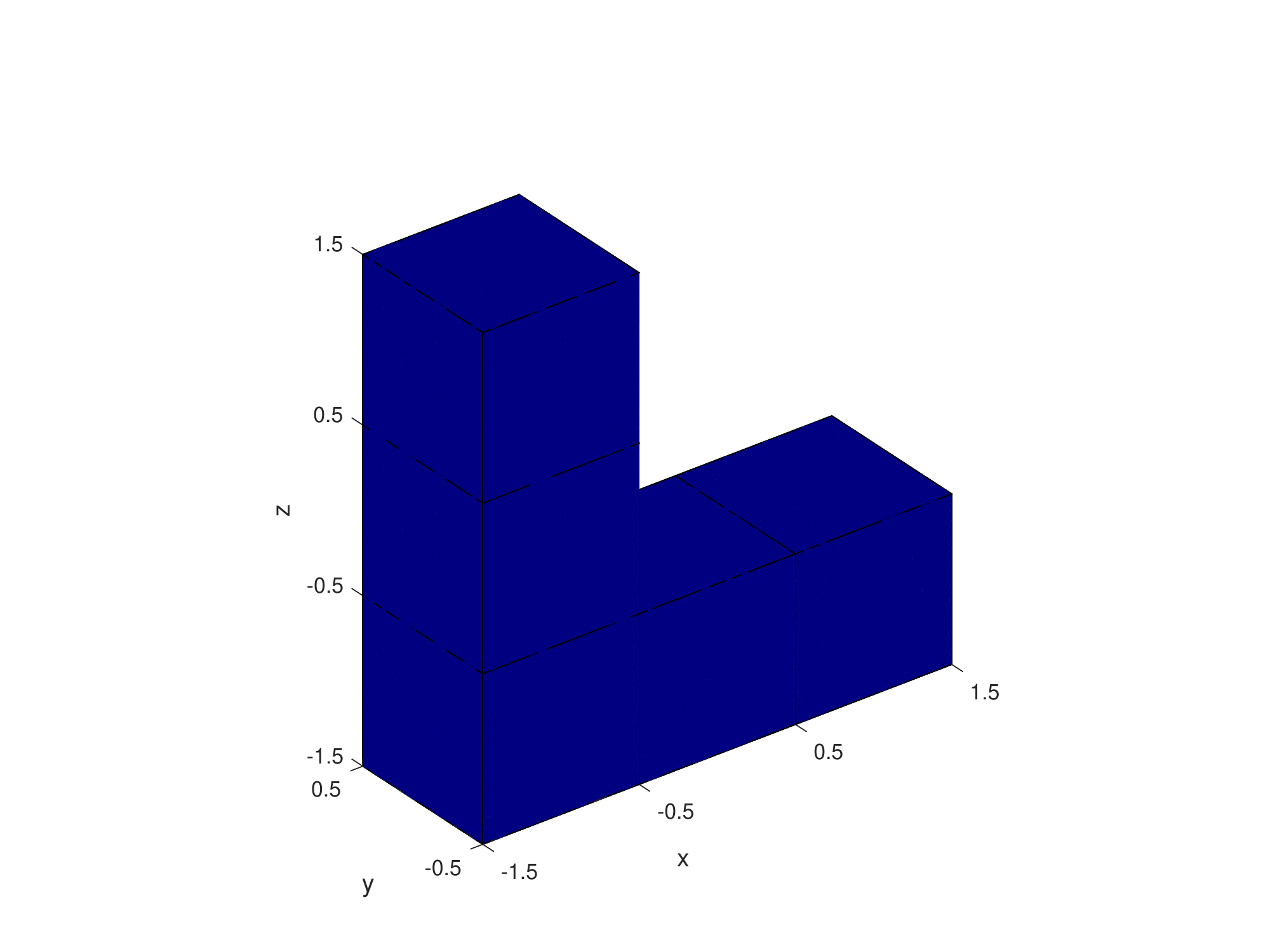}\hfill{}\includegraphics[width=0.48\textwidth]{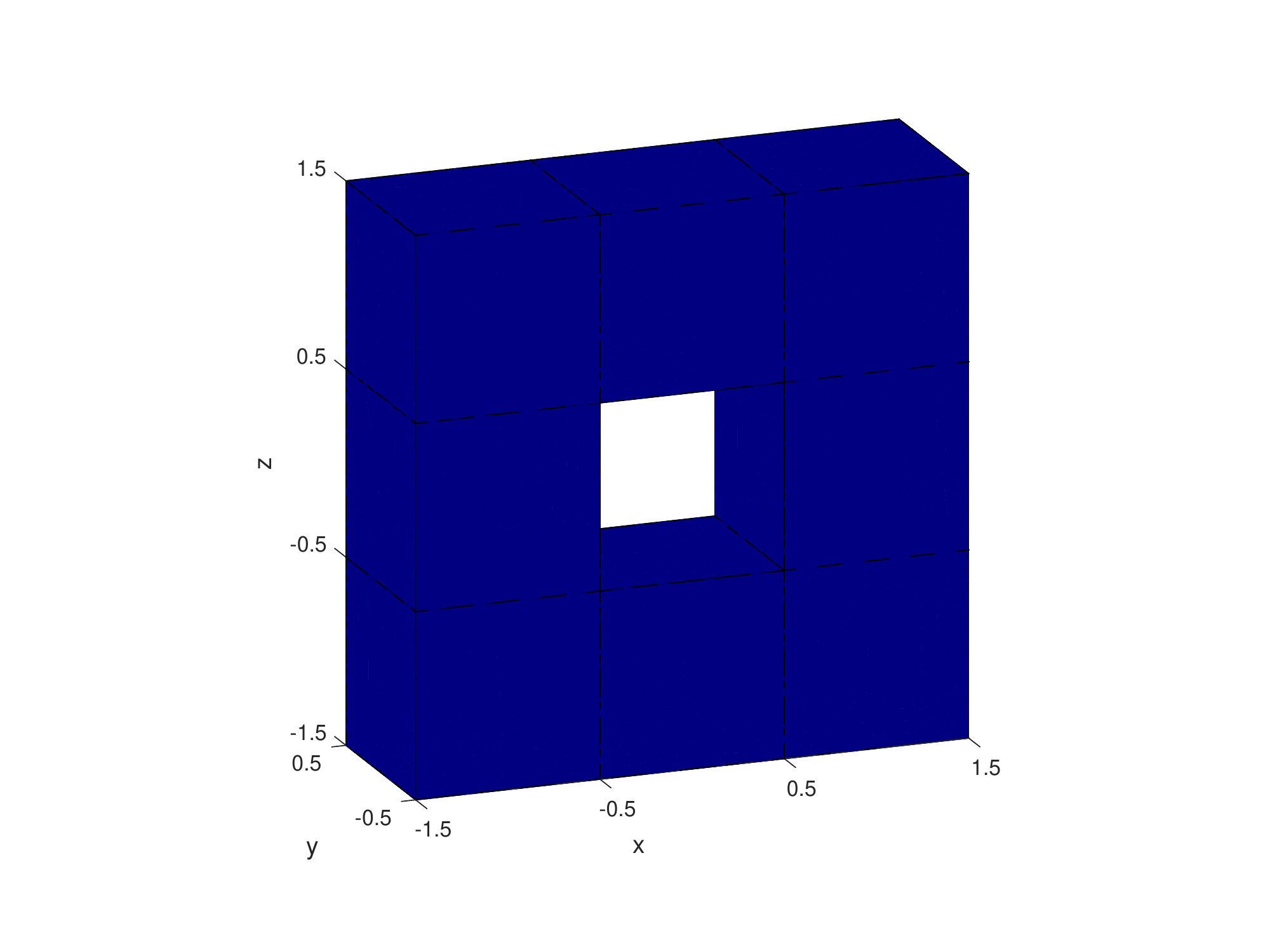}\hfill{}

\hfill{}(c)\hfill{}\hfill{}
(d)\hfill{}

\hfill{}\includegraphics[width=0.48\textwidth]{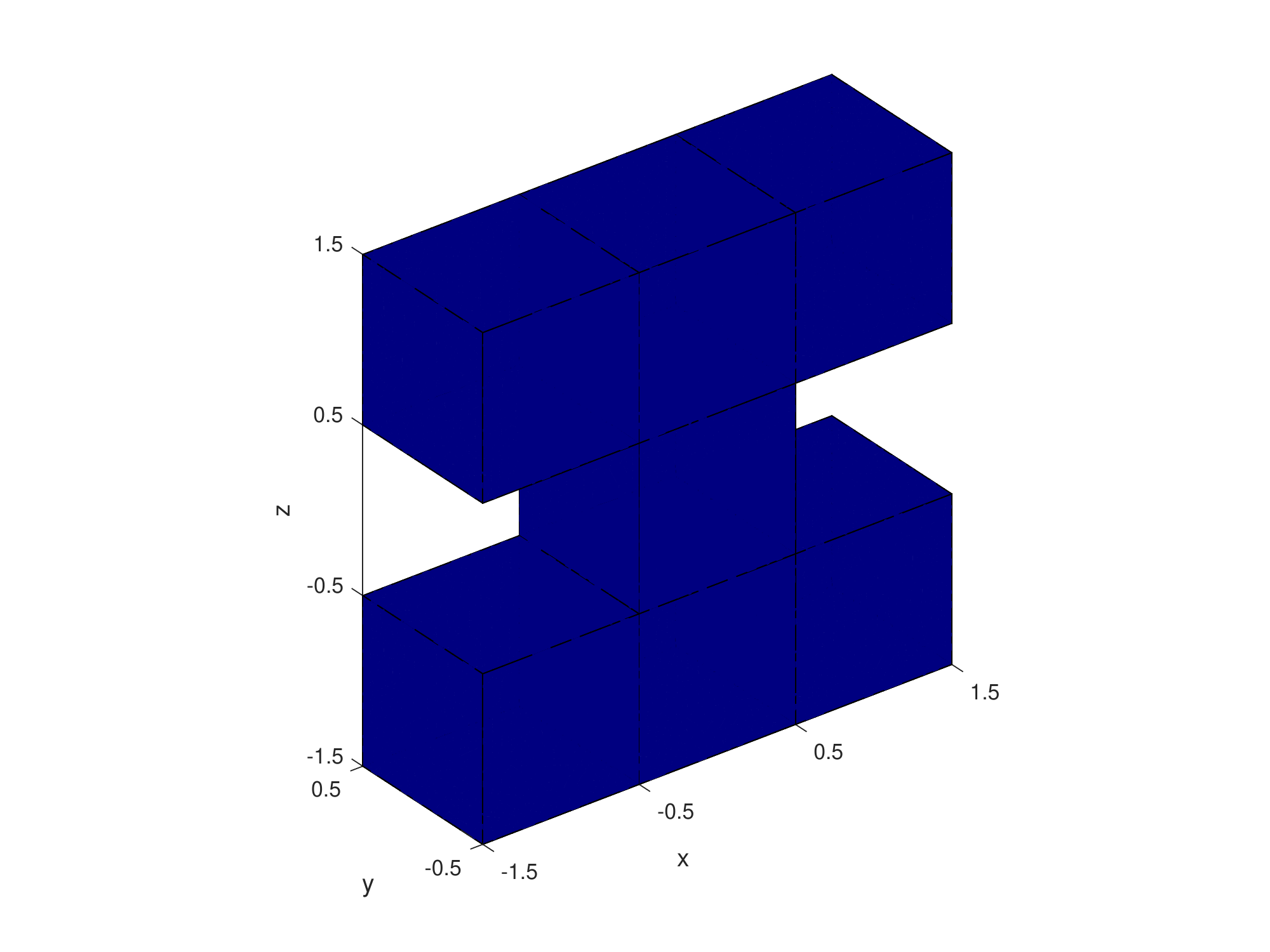}\hfill{}\includegraphics[width=0.48\textwidth]{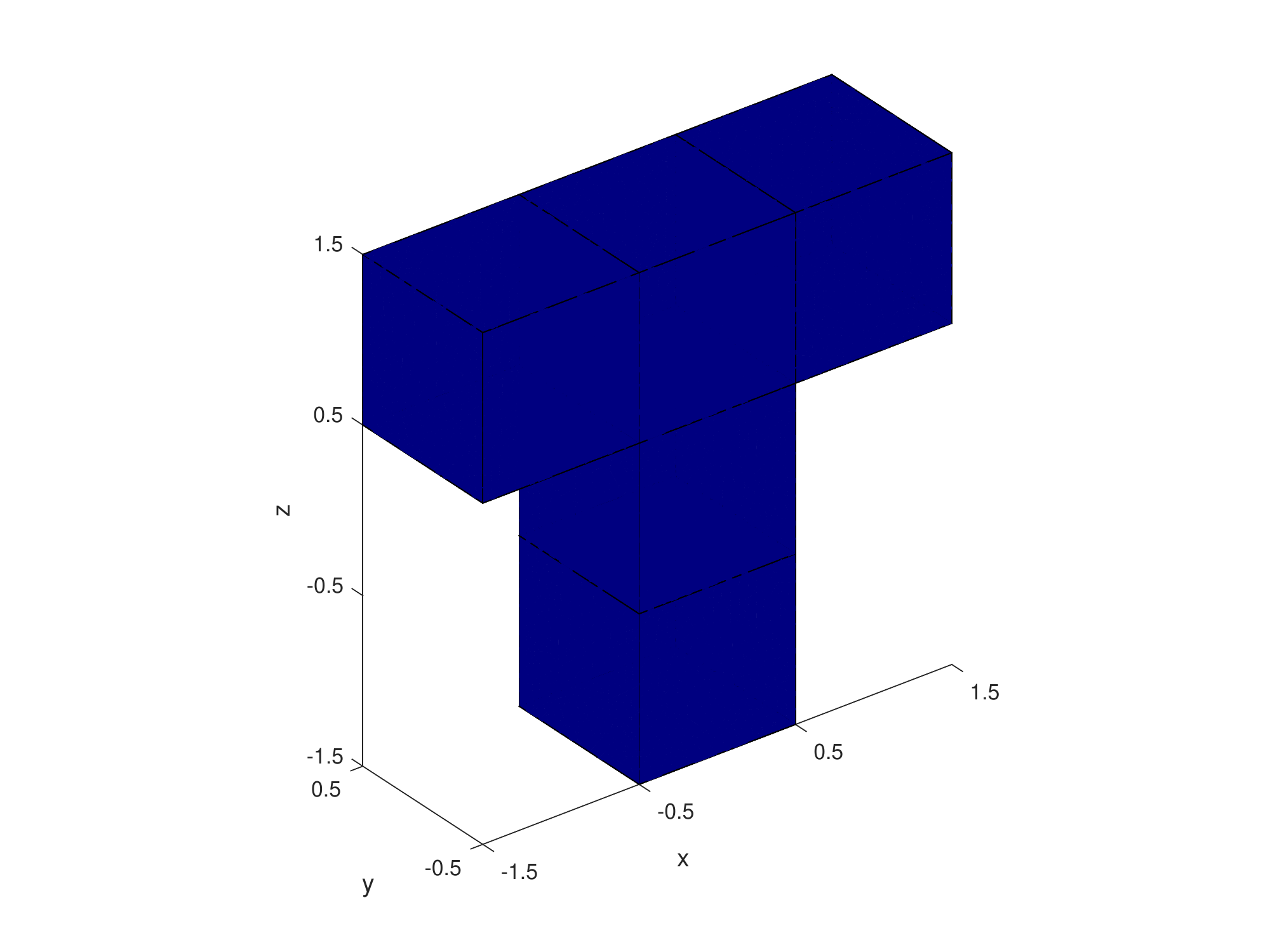}\hfill{}

\hfill{}(e)\hfill{}\hfill{}
(f)\hfill{}

\caption{\label{fig:Dictionary}Dictionary}

\end{figure}

The experimental setup is as follows.
As shown in Figure~\ref{fig:Dictionary}, we consider six gesture domains, $D_i$, $i=1,\ldots,6$, which are composed of a number of unit cubes, more exactly ranging from four to eight.
The measurement surface $\Gamma$ is set to be a unit square in the $x^2 x^3$-plane and centered at the origin. The respective scattered far fields on the unit sphere and near fields on the measurement surface $\Gamma$ of gesture domains
in the dictionary  $\mathfrak{D}$ as in \eqref{eq:dic:class} are first collected in advance for incident plane waves with different directions.

In all the examples, the electric far-field pattern $E_k^\infty(D_i, z; \hat{x})$, or abbreviated by $E_k^\infty(\hat{x})$ in this part,  is observed at 590
Lebedev quadrature points distributed on the unit sphere
$\mathbb{S}^{2}$ (cf.~\cite{Leb99} and references therein).  The exact far-field data
$E_k^\infty(\hat{x})$ are corrupted point-wise by the formula
\begin{equation}
E_{k,\delta}^\infty(\hat{x})  = E_k^\infty(\hat{x}) +\delta\zeta_1\underset{\theta}{\max}|E_k^\infty(\hat{x})|\exp(i2\pi
\zeta_2)\,,
\end{equation}
where $\delta$ refers to the relative noise level, and both  $\zeta_1$ and $\zeta_2$ follow the
uniform distribution ranging from $-1$ to $1$.   The scattered
electromagnetic fields are synthesized using the quadratic edge
element discretization in a spherical domain centered at the origin enclosed by a
spherical PML layer to damp the reflection. Local adaptive
refinement techniques within the inhomogeneous scatterer are adopted to
enhance the resolution of the scattered field. The far-field data
are approximated by the integral equation representation \cite[p.~181,
Theorem~3.1]{PiS02}  using the numerical quadrature. We refine the mesh successively
till the relative maximum error of successive groups of far-field
data is below $0.1\%$. The far-field patterns on the finest mesh
are used as the exact data. The near field data are obtained in a similar manner
except on the measurement surface $\Gamma$.

The target gesture is given by $\Omega:=D_{i}+z_{0}$, where $z_{0}$ is fixed to be $(150,\,0,\,0)$.   The wavelength of the low-frequency detecting wave for locating the position is set
to be $\lambda_{1}:=20$ and the wavelength of the high-frequency detecting wave for the shape identification
is set to be $\lambda_{2}:=2$.  In practical implementations, for the location determination, the far fields are approximated by the scattered fields that are measured more than ten wavelength away and
the indicator function are computed on the measurement surface $\Gamma$ in \eqref{eq:indicator:gamma} instead of the full aperture  $\mathbb{S}^2$ in \eqref{eq:indicator}. And for the shape determination, we use near fields data as
in \eqref{eq:indicator:J:gamma} from Theorem~\ref{thm:shape:deter:bdry}.

\subsection{PEC Gestures}

In the first example, we test with the six gesture domains with the PEC boundary condition.
The positions are first found in the location determination stage by locating the maximum value of the indicator function \eqref{eq:indicator:gamma}.
In the noise-free case,    the coordinates and
the distance from the exact location are shown in Table~\ref{tab:location-test-pec} for each gesture in the dictionary.
Compared with the exact position $z_0$, the difference between the exact and estimated positions are always below $0.1\%$ in terms of the Euclidean distance.

Next, we compute indicator function value in \eqref{eq:indicator:J:gamma} using the near field data measured on $\Gamma$ and the approximate position found in Table~\ref{tab:location-test-pec}.
The result of gesture recognition is shown in Table~\ref{tab:gesture-test-pec}. The values of the indicator function value have  been
rescaled between $0$ and $1$ by normalizing with respect to the maximum function value among all six reference gestures in each row of the table to highlight the unique gesture
identified. The same normalization procedure is employed in the sequel. We see from Table~\ref{tab:gesture-test-pec} that the peak value are always taken in the diagonal line
when the measurement data match with the precomputed data of the correct gesture.

\begin{table}
\hfill{}%
\begin{tabular}{|c|c|c|c|c|c|c|}
\hline
 & $D_{1}$ & $D_{2}$ & $D_{3}$ & $D_{4}$ & $D_{5}$ & $D_{6}$\tabularnewline
\hline
\hline
$\mathring{z}_{0}^{1}$ & $149.9838$ & $149.9611$ & $149.9742$ & $149.9632$ & $150.0075$ & $149.9853$\tabularnewline
\hline
$\mathring{z}_{0}^{2}$ & $0.0400$ & $0.0280$ & $-0.0096$ & $0.0442$ & $-0.0440$ & $0.0321$\tabularnewline
\hline
$\mathring{z}_{0}^{3}$ & $-0.0131$ & $-0.0110$ & $-0.0404$ & $-0.0456$ & $-0.0265$ & $-0.0485$\tabularnewline
\hline
$\left|\mathring{z}_{0}-z_{0}\right|$ & $0.0451$ & $0.0492$ & $0.0489$ & $0.0734$ & $0.0519$ & $0.0600$\tabularnewline
\hline
\end{tabular}\hfill{}

\caption{\label{tab:location-test-pec}PEC location test.}
\end{table}
\begin{table}
\hfill{}%
\begin{tabular}{|c|c|c|c|c|c|c|}
\hline
 & $D_{1}$ & $D_{2}$ & $D_{3}$ & $D_{4}$ & $D_{5}$ & $D_{6}$\tabularnewline
\hline
\hline
$D_{1}$ & $\boldsymbol{1.0000}$ & $0.9234$ & $0.9291$ & $0.8660$ & $0.8249$ & $0.9453$\tabularnewline
\hline
$D_{2}$ & $0.9233$ & $\boldsymbol{1.0000}$ & $0.9245$ & $0.9502$ & $0.9109$ & $0.9146$\tabularnewline
\hline
$D_{3}$ & $0.9290$ & $0.9242$ & $\boldsymbol{1.0000}$ & $0.9040$ & $0.9494$ & $0.9851$\tabularnewline
\hline
$D_{4}$ & $0.8660$ & $0.9510$ & $0.9040$ & $\boldsymbol{1.0000}$ & $0.9301$ & $0.9742$\tabularnewline
\hline
$D_{5}$ & $0.8249$ & $0.9107$ & $0.9492$ & $0.9300$ & $\boldsymbol{1.0000}$ & $0.9159$\tabularnewline
\hline
$D_{6}$ & $0.9451$ & $0.9147$ & $0.9849$ & $0.9732$ & $0.9149$ & $\boldsymbol{1.0000}$\tabularnewline
\hline
\end{tabular}\hfill{}

\caption{\label{tab:gesture-test-pec}PEC gesture test.}
\end{table}

In the noisy case with noise level of $5\%$,  the positions found in the first stage are shown in Table \ref{tab:location-test-pec-noise}.
the difference between the exact and estimated positions is still very small.
The result of gesture recognition is shown in Table \ref{tab:gesture-test-pec-noise}, which clearly shows that  all the correct pairs matches the best.
The test with noisy data shows the robustness with respect to
noisy measurement data  of both the locating and recognition indicator functions in \eqref{eq:indicator:gamma} and \eqref{eq:indicator:J:gamma}. This salient robustness
is due to the inner product operation, which eliminates implicitly the noisy part in light of the orthogonality.

\begin{table}
\hfill{}%
\begin{tabular}{|c|c|c|c|c|c|c|}
\hline
 & $D_{1}$ & $D_{2}$ & $D_{3}$ & $D_{4}$ & $D_{5}$ & $D_{6}$\tabularnewline
\hline
\hline
$\mathring{z}_{0}^{1}$ & $150.0278$ & $150.0547$ & $150.0965$ & $150.0158$ & $150.0971$ & $150.0958$\tabularnewline
\hline
$\mathring{z}_{0}^{2}$ & $0.0679$ & $0.0743$ & $0.0655$ & $0.0706$ & $0.0277$ & $0.0097$\tabularnewline
\hline
$\mathring{z}_{0}^{3}$ & $0.0758$ & $0.0392$ & $0.0171$ & $0.0032$ & $0.0046$ & $0.0823$\tabularnewline
\hline
$\left|\mathring{z}_{0}-z_{0}\right|$ & $0.1055$ & $0.1003$ & $0.1173$ & $0.1196$ & $0.0322$ & $0.1277$\tabularnewline
\hline
\end{tabular}\hfill{}

\caption{\label{tab:location-test-pec-noise}PEC location test with $5\%$ noise.}
\end{table}
\begin{table}
\hfill{}%
\begin{tabular}{|c|c|c|c|c|c|c|}
\hline
 & $D_{1}$ & $D_{2}$ & $D_{3}$ & $D_{4}$ & $D_{5}$ & $D_{6}$\tabularnewline
\hline
\hline
$D_{1}$ & $\boldsymbol{1.0000}$ & $0.9453$ & $0.9632$ & $0.8213$ & $0.9182$ & $0.9649$\tabularnewline
\hline
$D_{2}$ & $0.9431$ & $\boldsymbol{1.0000}$ & $0.9374$ & $0.8864$ & $0.9205$ & $0.9061$\tabularnewline
\hline
$D_{3}$ & $0.9651$ & $0.9255$ & $\boldsymbol{1.0000}$ & $0.9213$ & $0.9070$ & $0.9124$\tabularnewline
\hline
$D_{4}$ & $0.8268$ & $0.8811$ & $0.9219$ & $\boldsymbol{1.0000}$ & $0.9621$ & $0.9450$\tabularnewline
\hline
$D_{5}$ & $0.9152$ & $0.9213$ & $0.9071$ & $0.9491$ & $\boldsymbol{1.0000}$ & $0.9378$\tabularnewline
\hline
$D_{6}$ & $0.9649$ & $0.9066$ & $0.9123$ & $0.9459$ & $0.9367$ & $\boldsymbol{1.0000}$\tabularnewline
\hline
\end{tabular}\hfill{}

\caption{\label{tab:gesture-test-pec-noise}PEC gesture test with $5\%$ noise.}
\end{table}

\subsection{Medium Case}

In the second example, we test with an inhomogeneous medium among the six reference gesture domains.
\[
n_{k,\,\Omega}=\begin{cases}
1 & x\in\mathbb{R}^{3}\backslash\bar{\Omega}\\
5 & x\in\Omega
\end{cases}.
\]

The results of location and gesture tests are shown, respectively, in Tables~\ref{tab:location-test-medium} and \ref{tab:gesture-test-medium} for the noise-free case,
and in Tables~\ref{tab:location-test-medium-noise} and \ref{tab:gesture-test-medium-noise} for the noisy case with $5\%$ noise level.
Both noise-free and noisy cases tell us that our locating and gesture recognition algorithms are very robust with noise and work very well even with data of limited aperture.
And the computational efforts is quite less and the recognition schemes are very efficient only involving with inner product by known data at hand.

\begin{table}
\hfill{}%
\begin{tabular}{|c|c|c|c|c|c|c|}
\hline
 & $D_{1}$ & $D_{2}$ & $D_{3}$ & $D_{4}$ & $D_{5}$ & $D_{6}$\tabularnewline
\hline
\hline
$\mathring{z}_{0}^{1}$ & $150.0417$ & $150.0254$ & $149.9576$ & $150.0279$ & $150.0069$ & $149.9837$\tabularnewline
\hline
$\mathring{z}_{0}^{2}$ & $-0.0214$ & $-0.0120$ & $-0.0446$ & $0.0434$ & $-0.0031$ & $-0.0338$\tabularnewline
\hline
$\mathring{z}_{0}^{3}$ & $0.0257$ & $0.0068$ & $0.0031$ & $-0.0370$ & $-0.0488$ & $0.0294$\tabularnewline
\hline
$\left|\mathring{z}_{0}-z_{0}\right|$ & $0.0535$ & $0.0289$ & $0.0616$ & $0.0635$ & $0.0494$ & $0.0477$\tabularnewline
\hline
\end{tabular}\hfill{}

\caption{\label{tab:location-test-medium}Medium location test.}
\end{table}
\begin{table}
\hfill{}%
\begin{tabular}{|c|c|c|c|c|c|c|}
\hline
 & $D_{1}$ & $D_{2}$ & $D_{3}$ & $D_{4}$ & $D_{5}$ & $D_{6}$\tabularnewline
\hline
\hline
$D_{1}$ & $\boldsymbol{1.0000}$ & $0.9311$ & $0.8848$ & $0.9393$ & $0.8716$ & $0.9418$\tabularnewline
\hline
$D_{2}$ & $0.9312$ & $\boldsymbol{1.0000}$ & $0.9174$ & $0.8838$ & $0.9100$ & $0.9086$\tabularnewline
\hline
$D_{3}$ & $0.8834$ & $0.9179$ & $\boldsymbol{1.0000}$ & $0.9295$ & $0.9740$ & $0.8854$\tabularnewline
\hline
$D_{4}$ & $0.9398$ & $0.8899$ & $0.9004$ & $\boldsymbol{1.0000}$ & $0.9200$ & $0.9864$\tabularnewline
\hline
$D_{5}$ & $0.8737$ & $0.9171$ & $0.9422$ & $0.9183$ & $\boldsymbol{1.0000}$ & $0.9131$\tabularnewline
\hline
$D_{6}$ & $0.9346$ & $0.9087$ & $0.8557$ & $0.9131$ & $0.9089$ & $\boldsymbol{1.0000}$\tabularnewline
\hline
\end{tabular}\hfill{}

\caption{\label{tab:gesture-test-medium}Medium gesture test.}
\end{table}

\begin{table}
\hfill{}%
\begin{tabular}{|c|c|c|c|c|c|c|}
\hline
 & $D_{1}$ & $D_{2}$ & $D_{3}$ & $D_{4}$ & $D_{5}$ & $D_{6}$\tabularnewline
\hline
\hline
$\mathring{z}_{0}^{1}$ & $149.9758$ & $149.9762$ & $149.9722$ & $149.9819$ & $149.9586$ & $149.9529$\tabularnewline
\hline
$\mathring{z}_{0}^{2}$ & $-0.0091$ & $0.0103$ & $-0.0383$ & $-0.0076$ & $-0.0238$ & $0.0429$\tabularnewline
\hline
$\mathring{z}_{0}^{3}$ & $0.0095$ & $0.0211$ & $-0.0203$ & $0.0008$ & $0.0301$ & $0.0230$\tabularnewline
\hline
$\left|\mathring{z}_{0}-z_{0}\right|$ & $0.0275$ & $0.0334$ & $0.0515$ & $0.0197$ & $0.0565$ & $0.0677$\tabularnewline
\hline
\end{tabular}\hfill{}

\caption{\label{tab:location-test-medium-noise}Medium location test with $5\%$ noise.}
\end{table}
\begin{table}
\hfill{}%
\begin{tabular}{|c|c|c|c|c|c|c|}
\hline
 & $D_{1}$ & $D_{2}$ & $D_{3}$ & $D_{4}$ & $D_{5}$ & $D_{6}$\tabularnewline
\hline
\hline
$D_{1}$ & $\boldsymbol{1.0000}$ & $0.8800$ & $0.9109$ & $0.9321$ & $0.8895$ & $0.9280$\tabularnewline
\hline
$D_{2}$ & $0.8738$ & $\boldsymbol{1.0000}$ & $0.8666$ & $0.9520$ & $0.9303$ & $0.9032$\tabularnewline
\hline
$D_{3}$ & $0.9105$ & $0.8656$ & $\boldsymbol{1.0000}$ & $0.9095$ & $0.8817$ & $0.9021$\tabularnewline
\hline
$D_{4}$ & $0.9320$ & $0.9221$ & $0.9802$ & $\boldsymbol{1.0000}$ & $0.9160$ & $0.9287$\tabularnewline
\hline
$D_{5}$ & $0.8863$ & $0.9285$ & $0.8819$ & $0.9162$ & $\boldsymbol{1.0000}$ & $0.9169$\tabularnewline
\hline
$D_{6}$ & $0.9279$ & $0.9030$ & $0.9256$ & $0.9141$ & $0.9171$ & $\boldsymbol{1.0000}$\tabularnewline
\hline
\end{tabular}\hfill{}

\caption{\label{tab:gesture-test-medium-noise}Medium gesture test with $5\%$ noise.}
\end{table}

\section{Conclusion}\label{sect:conclusion}

We proposed and analyzed the gesture recognition with electromagnetic detection. The mathematical principle is deeply involved with inverse scattering theory of electromagnetic waves. We employed the translation relations between the scattering by electromagnetic incident point source and the scattering by incident plane waves. A two-stage recognition algorithm is designed based on the theoretical analysis. In the first stage low frequency scattering data are employed for locating positions, and in the second stage the regular frequency scattering data are employed for determining the shapes. The data needed for the recognition algorithm are stored in a precomputed dictionary, and computations of the proposed algorithm are mainly inner products of the corresponding data, which are very robust to noise. Various numerical tests also show the efficiency of the proposed algorithm. There are still some interesting problems to solve, i.e., determining the information of the rotations of the scatterers, since these information is already included in the the scattered data dictionary generally.

\section*{Acknowledgement}
{\small
The work of J. Li was supported by the NSF of China under the grant No.\, 11571161, the Shenzhen Sci-Tech Fund No. JCYJ20160530184212170 and the SUSTech startup fund.
The work of H. Liu was supported by the FRG grants and startup fund from Hong Kong Baptist University, and Hong Kong RGC General Research Funds (12302415 and 12302017).
  Hongpeng Sun acknowledges the support of
Fundamental Research Funds for the Central Universities, and the
research funds of Renmin University of China (15XNLF20). He also acknowledges discussion with Dr. Yuliang Wang during working on the topic.
}


\begin{thebibliography}{99}


\bibitem{AA1} {H. Ammari, T. Boulier, J. Garnier}, {\it Modeling active electrolocation in weakly
electric fish}, SIAM J. Imaging Sci., 6 (2013), pp. 285--321.

\bibitem{AA2} {H. Ammari, T. Boulier, J. Garnier, W. Jing, H. Kang, H. Wang}, {\it Target identi
cation using dictionary matching of generalized polarization tensors}, Found. Comput.
Math., 14(2014), pp 27--62.

\bibitem{AA3} {H. Ammari, M. Tran, H. Wang}, {\it Shape identification and classification in echolo
cation}, SIAM J. Imaging Sci., 7(3), (2014), pp. 1883--1905.

\bibitem{BW} {M. Born, E. Wolf}, {\it Princinples of Optics Electromagnetic Theory of Propagation, Interference and Diffraction of Light},
7th Edition, Cambridge University Press 1999.
\bibitem{CK} {D. Colton R. Kress}, {\it Inverse Acoustic and Electromagnetic Scattering Theory}, Springer, Applied Mathematical Sciences Vol.93, Third Edition, 2013.

\bibitem{DR} {G. Dassios, R. Kleinman},  {\it Low Frequency Scattering}, Clarendon Press, Oxford, 2000.
\bibitem{Kir1} { A. Kirsch, F. Hettlich}, {\it The Mathematical Theory of Time-Harmonic Maxwell's Equations

Expansion-, Integral-, and Variational Methods}, Applied Mathematical Sciences, Vol. 190, Springer 2015.


\bibitem{EBNT} {A. Erol, G. Bebis, M. Nicolescu, R. D. Boyle, X. Twombly}, {\it Vision-based hand pose estimation: a review}, Computer Vision and Image Understanding 108 (2007) pp. 52--73.




\bibitem{KG} {A. Kirsch, N. Grinberg}, {\it The Factorization Method for Inverse Problems}, Oxford
Lecture Series in Mathematics and its Applications, Oxford University Press,
Oxford, 2008.

\bibitem{KT} {M. Kolsch, M. Turk}, {\it Fast 2D Hand Tracking with Flocks of Features and Multi-Cue Integration}, IEEE Computer Vision and Pattern Recognition Workshop, 2004. CVPRW 04.


\bibitem{Leb99}
\textsc{V.~I.~Lebedev and D.~N.~Laikov}, {\em A quadrature formula for
the sphere of the 131st algebraic order of accuracy}, Doklady
Mathematics,  {\bf 59} (1999), pp.~477--481.


\bibitem{LLZ} {J. Li, H. Liu and J. Zou}, {\it Locating multiple multiscale acoustic scatterers}, SIAM Multiscale Model. Simul., 12(3) (2014), pp. 927--952.


\bibitem{LWY} {H. Liu, Y. Wang, C. Yang}, {\it Mathematical design of a novel gesture-based instruction/input device using wave detection}, SIAM J. Imaging Sci., 9(2) (2016), pp. 822--841.


\bibitem{MKA} {A. Makris, N. Kyriazis. A. A. Argyros}, {\it Hierarchical particle filtering for 3D hand tracking},
IEEE Conference on Computer Vision and Pattern Recognition Workshops (CVPRW), 2015.

\bibitem{PM} {P. Monk}, {\it Finite Element Methods for Maxwell's Equations}, Clarendon Press, Oxford, 2003.
2003

\bibitem{ND} {J. C. N\'{e}d\'{e}lec}, {\it Acoustic and Electromagnetic Equations: Integral Representations
for Harmonic Problems}, Springer-Verlag, New York, 2001.

\bibitem{PST} {V. I. Pavlovic, R. Sharma, T. S. Huang}, {\it Visual interpretation of hand gestures for human-computer interaction: a review}, IEEE Transactions on Pattern Analysis and Machine Intelligence, 1997.



\bibitem{PiS02}
\textsc{ R.~Pike and P.~Sabatier eds.}, {\it Scattering: Scattering and
Inverse Scattering in Pure and Applied Science}, Academic Press,
2002.

\bibitem{RP} {R. Potthast}, {\it Point Sources and Multipoles in Inverse Scattering Theory}, Chapman \& Hall/CRC Research Notes in Mathematics, No. 427, 2001.

\bibitem{GP} {Project Soli}, {\it Google ATAP, https://www.google.com/atap/project-soli/}.



\end{thebibliography}
\end{document}